\documentclass[a4paper, 12pt]{article}
\usepackage{mathpaper}
\usepackage{mathshort}
\allowbreak

\title{On the rate of exponential decay of coefficients on homogeneous spaces}
\author{Yves Benoist and Siwei Liang}
\date{}

\begin{document}
    \maketitle
    \begin{abstract}
        For any homogeneous space of a noncompact semisimple Lie group $G$, we define an exponent with multiple interpretations from representation theory and group theory. As an application, we give a temperedness criterion for $\Lspace^2(G/H)$ for any closed subgroup $H$ of $G$, which extends the existing ones of Benoist--Kobayashi for connected subgroups and Lutsko--Weich--Wolf for discrete subgroups.
    \end{abstract}
    \tableofcontents

\section{Introduction}\label{sec intro}

Let $G$ be a connected noncompact semisimple Lie group with finite center, $K$ be a maximal compact subgroup of $G$, and $H$ be a closed subgroup of $G$.
The homogeneous space $G/H$ admits a $G$-quasi-invariant Radon measure, giving rise to the unitary representation $\lambdagh$ of $G$ on the Hilbert space $\Ltwogh$ of square-integrable functions on $G/H$ by left translation. These are called the \emph{quasi-regular representations} of $G$.

\subsection{Motivations}
\subsubsection*{The Plancherel formula}
The Plancherel formula for a unitary representation of $G$ is the explicit decomposition into irreducible ones.
This line of study traces back to the pioneering work of Cartan, Weyl, Gelfand, and Harish-Chandra early in the 20th century.
Now thanks to the tremendous input of many other mathematicians, the Plancherel formulae for certain classes of quasi-regular representations $\Lspace^2(G/H)$ are known, including the ones on $G$ itself, the Riemannian/affine symmetric spaces, and the real spherical spaces, cf.\,\cite{ban-schlichtkrull2005the-plancherel, delorme-knop2021plancherel} and the references therein.

However, for more general homogeneous spaces $G/H$, the aforementioned approaches hardly generalize, and instead of pursuing the full decomposition, one may ask more accessible questions on the support of the $\Lspace^2(G/H)$, e.g.\,what are the irreducible representations it may contain. 

\subsubsection*{Tempered homogeneous spaces}
One line of such research, intiated by Benoist--Kobayashi, is to study when is the quasi-regular representation $\Lspace^2(G/H)$ a tempered representation, or equivalently when is the support of $\Lspace^2(G/H)$ contained in that of $\Lspace^2(G)$. In a series of work \cite{benoist-kobayashi2015tempered, benoist-kobayashi2022tempered2, benoist-kobayashi2021tempered3, benoist-kobayashi2023tempered4}, they established a simple geometric criterion for those $H$ with finitely many connected components: $\Lspace^2(G/H)$ is tempered iff the Lie algebras $\glie, \hlie$ of $G,H$ satisfy a growth condition.

In a complementary direction, the same question for discrete subgroups have been studied by other authors.
The case of $G$ being simple of real rank one is well known from hyperbolic geometry: for a discrete subgroup $\Gamma$, $\Lspace^2(G/\Gamma)$ is tempered iff the Laplacian on the locally symmetric space $K\backslash G/\Gamma$ has maximal spectral gap.
The latter condition is known equivalent to that the critical exponent of $\Gamma$ does not exceed half of that of $G$, by the accumulative works of Elstrodt \cite{elstrodt1973die-resolvente}, Patterson \cite{patterson1976the-limit}, Sullivan \cite{sullivan1987related}, Corlette \cite{corlette1990hausdorff}, Shalom \cite{shalom2000rigidity}, and Leuzinger \cite{leuzinger2003critical}.

When the semisimple Lie group $G$ has higher rank, this result has been extended more recently, first by Edwards--Oh \cite{edwards-oh2023temperedness} and Lee--Oh \cite{lee-oh2024dichotomy} for Anosov subgroups and then by Lutsko--Weich--Wolf \cite{lutsko-weich2024polyhedral} for all discrete subgroups, to a similar growth criterion on the temperedness of $\Lspace^2(G/\Gamma)$, with the critical exponent replaced by Quint's growth indicator function $\psi_\Gamma$ introduced in \cite{quint2002divergence}.

\begin{center}
    \emph{The goal of this paper is to unify these geometric criteria for the temperedness of $\Lspace^2(G/H)$ and extend them quantitatively to all homogeneous spaces of $G$, i.e.\,to arbitrary closed subgroups $H$.}
\end{center}

\subsection{Notation and reformulation}\label{subsec: notation and reformulation}
We fix some notation and then introduce four exponents related to the homogeneous space $G/H$, which will be the main protagonists of our paper.
They will allow us to rephrase the aforementioned temperedness criteria.

Unless otherwise stated, the Lie algebra of a Lie group will be denoted by the corresponding Fraktur letter.
Let $\theta: \glie\to \glie$ be a Cartan involution which fixes $\klie$ and the associated Cartan decomposition be $\glie = \klie \oplus \plie$. 
Let $\alie$ be a Cartan subspace in $\plie$ (a maximal real split abelian subalgebra). The adjoint action of $\alie$ on $\glie$ is jointly diagonalizable over $\bbR$, giving rise to the root system $\Sigma$ with
$\glie = \glie_0 \oplus \bigoplus_{\alpha\in\Sigma} \glie_\alpha$.
Fix a positive system $\Sigma^+$. The real linear form
$$ \rho:= \frac{1}{2} \sum_{\alpha\in\Sigma^+}(\dim\glie_\alpha)\alpha \in \alie'$$
is defined in the standard way and measures the complexity of the adjoint action.
The closed positive additive and multiplicative Weyl chambers  are respectively denoted by 
$\alie^+:= \setdef{X\in\alie: \alpha(X)\geq 0,\,\forall \alpha\in\Sigma^+}$ and $A^+:=\exp\alie^+$.
We obtain the Cartan decomposition $G=K A^+ K$, for which the $A^+$-component is unique, and thus the Cartan projection
$$ \kappa: G\to \alie^+, \;\; g\in K \exp\kappa(g) K,$$
which measures how far group elements are from the identity. We will write $\rho\kappa$ for the function $\rho\circ \kappa : G\to \bbR^{+}$ for brevity.

\begin{definition}\label{def: uniform decay exponent}
    The \emph{coefficient decay exponent} $\thetagh$ is defined as the infimum of $\theta \in [0,1]$ such that for any $f\in \Cc(G/H)$, there exists a constant $C=C(f)>0$ such that uniformly for all $g\in G$, we have
    \begin{equation*}
        \Modofinner{\lambdagh(g)f, f} \leq C  e^{2(\theta-1)\rho\kappa(g)}.
    \end{equation*}
\end{definition}

\noindent\emph{Motivation for $\thetagh$.}
By the Cowling--Haagerup--Howe theorem \cite{cowling-haagerup1988almost}, we have (\autoref{cor: temperedness iff thetagh})
\begin{equation}\label{eq thm: temperedness and thetagh}
    \Lspace^2(G/H) \textup{ is tempered} \iff \thetagh\leq 1/2.
\end{equation}

\begin{remark*}
    The number $\thetagh$ does not depend on the choice of the Cartan decomposition of $G$.
    For a discrete subgroup $\Gamma$, the exponent $\theta_{G/\Gamma}$ is related to the number $\theta_\Gamma(\rho)$ defined in \cite[\S 1]{lutsko-weich2024polyhedral} via
    \begin{equation*}
        \theta_\Gamma(\rho) = \maxof{2\theta_{G/\Gamma} - 1 ,\, 0}.
    \end{equation*}
\end{remark*}

\begin{definition}
    The \emph{optimal integrability exponent} $p_{G/H}$ is defined as the infimum of $p\in[1,\infty]$ such that for any $f\in\Cc(G/H)$, we have
    \begin{equation*}
        \Innerprod{\lambda_{G/H}(\cdot)f, f} \in \Lspace^p(G).
    \end{equation*}
\end{definition}

\noindent\emph{Motivation for $p_{G/H}$.} 
By \cite[Thm 1]{cowling-haagerup1988almost}, we have
\begin{equation}\label{eq thm: temperedness and pgh}
    \Ltwogh \textup{ is tempered} \iff p_{G/H} \leq 2.
\end{equation}
When $p_{G/H}> 2$, the equality $\thetagh = 1-1/p_{G/H}$ is essentially known by combining \cite[Thm 5.3]{samei-wiersma2024exotic} with \cite[Cor 4.3]{cowling2023decay}.

The exponent $p_{G/H}$ gives a measurement of the spectral gap of $\Lspace^2(G/H)$.
Indeed, as a consequence of the Kunze--Stein phenomenon \cite{cowling1978the-kunze-stein}, the theorem \cite[Thm 5.3]{samei-wiersma2024exotic} of Samei--Wiersma implies that, when $p_{G/H}\geq 2$, \textbf{all} the matrix coefficients of $\Lspace^2(G/H)$ are $\Lspace^{p_{G/H}+\varepsilon}$-functions for all $\varepsilon>0$, and moreover this property is satisfied by all those unitary representations that are weakly contained in $\Lspace^2(G/H)$. Hence, the integrability exponent $p_{G/H}$ controls the support of $\Lspace^2(G/H)$.

\begin{remark*}
    The optimal integrability exponent coincides with the number $p_{G/H}$ defined in \cite[\S 4.2]{benoist-kobayashi2015tempered} when $H$ is reductive  and with the number $q(G; G/H)$ defined in \cite[Def 7.12]{kobayashi2025proper} when $H$ is a unimodular subgroup.
\end{remark*}

\begin{definition}[\autoref{def: relative exponent of closed subgroups}]
    The \emph{relative volume growth exponent} $\deltagh$ of $H$ inside $G$, which turns out to lie in $[0,1]$, is defined by
    \begin{equation*}
        \deltagh := \maxde{0,\, \limsup_{g\to \infty} \frac{\log \nuh{H\cap BgB}}{\log \nug{BgB}}},
    \end{equation*}
    where $B$ is any compact subset of $G$ of nonempty interior. The measure $\nu_G$ is the Haar measure on $G$, while $\dif \nu_H(h) = \left(\det\Ad_{H}h\right)^{1/2} \dif h$ is the symmetric measure on $H$ (see \eqref{eq def: symmetric measure}).
\end{definition}
\noindent\emph{Motivation for $\deltagh$.} 
As a measurement of the exponential volume growth rate of $H$ but relative to the ambient group $G$, the exponent $\deltagh$ generalizes the critical exponents of discrete subgroups. As one would expect, it equals the abscissa of convergence for the following analogue of Poincaré series (\autoref{prop: deltagh as the abscissa of convergence})
    \begin{equation*}
        [0,\infty] \ni t \longmapsto \int_H e^{-2t\rho\kappa(h)} \dif\nu_H(h).
    \end{equation*}
For a discrete subgroup $\Gamma$, one recovers $\deltaggamma$ as the $2\rho$-directional critical exponent (\autoref{prop: delta = max of psi/rho, discrete subgroups}). Then \cite[Thm 1.1]{lutsko-weich2024polyhedral} is translated to
\begin{equation}\label{eq thm: lww}
    \maxof{\theta_{G/\Gamma}, \, 1/2} = \maxof{\deltaggamma, \, 1/2},
\end{equation}
which gives the criterion involving Quint's growth indicator function $\psi_\Gamma$:
\begin{equation}\label{eq thm: lww criterion}
    \Lspace^2(G/\Gamma) \textup{ is tempered} \iff \psi_\Gamma \leq \rho. 
\end{equation}

\begin{definition}[\autoref{def: betagh}]
    The \emph{local volume decay exponent} $\betagh\in[0,1]$ is defined through the Lie algebras $\glie, \hlie$ by
    \begin{equation*}
        \rhogh := \sup_{\hlie} \frac{\rhoh}{\rhog},
    \end{equation*}
    where the rho-functions $\rhoh, \rhog: \hlie \to \bbR^+$ are respectively defined as the half sum of absolute values of the real parts of complex eigenvalues for the adjoint action of $\hlie$ on the spaces $\hlie, \glie$, as in \cite[\S 2.3]{benoist-kobayashi2022tempered2}.
\end{definition}
\noindent\emph{Motivation for $\betagh$.}
The rho-functions were used by Benoist--Kobayashi in \cite{benoist-kobayashi2015tempered} to capture the exponential volume decay rate from algebraic data (\autoref{cor: rho fun and volume decay}). They essentially proved for any reductive subgroup $H$,
\begin{equation}\label{eq thm: bk equality}
    \thetagh = \betagh = 1 - \frac{1}{p_{G/H}}.
\end{equation}
In \cite{benoist-kobayashi2022tempered2}, they extended the temperedness criterion to the statement that for any closed subgroup $H$ with finitely many connected components,
\begin{equation}\label{eq thm: bk criterion}
    \Lspace^2(G/H) \textup{ is tempered} \iff \betagh \leq 1/2.
\end{equation}

Our main results establish precise quantitative relations among the four exponents $\thetagh, \deltagh, p_{G/H}, \betagh$ that we have just defined.

\subsection{Statement of main results}
Let $G$ be a noncompact real semisimple algebraic group. 
Our first main result contains a response to the optimal integrability problem \cite[Prob 7.13]{kobayashi2025proper} for all homogeneous spaces of $G$.

\begin{mytheorem}\label{mythm: theta and delta}
    Let $H$ be a closed subgroup of $G$. Then $$\thetagh = \deltagh = 1- \frac{1}{p_{G/H}} \geq \betagh.$$
\end{mytheorem}

Our proof of the equality $\thetagh=\deltagh$ is inspired by the method of \cite{lutsko-weich2024polyhedral} which we generalize to all induced representations.

As an immediate consequence of \autoref{mythm: theta and delta} and the uniform decay characterization \eqref{eq thm: temperedness and thetagh} of temperedness, we obtain the following temperedness criterion, in response to \cite[Prob 7.18]{kobayashi2025proper}.
\begin{mycorollary}\label{mythm: temperedness and delta}
    Let $H$ be a closed subgroup of $G$. Then
    \begin{equation*}
        \Ltwogh \textup{ is tempered} \iff \deltagh \leq 1/2.
    \end{equation*}
\end{mycorollary}

In addition, \autoref{mythm: theta and delta} unifies the results \eqref{eq thm: bk equality} of Benoist--Kobayashi and \eqref{eq thm: lww} of Lutsko--Weich--Wolf for the following reason: in special cases, one can easily relate the relative volume growth exponent $\deltagh$ to other existing quantities and obtain the following corollaries.
\begin{mycorollary}\label{mythm: reductive subgp, delta=theta=rho}
    Let $H$ be a reductive subgroup of $G$. Then
    \begin{equation*}
        \thetagh = \deltagh = \rhogh.
    \end{equation*}
\end{mycorollary}

\begin{mycorollary}\label{mythm: discrete subgp, delta=theta=}
    Let $\Gamma$ be a discrete subgroup of $G$. Then
    \begin{equation*}
        \theta_{G/\Gamma} = \deltaggamma = \maxde{\sup_{\alie^+} \frac{\psi_\Gamma}{2\rho},\, 0}.
    \end{equation*}
\end{mycorollary}

Our second main result, concerning closed subgroups with finitely many connected components, gives a quantitative extension of the temperedness criterion \eqref{eq thm: bk criterion} of Benoist--Kobayashi.
\begin{myproposition}\label{mythm: rho theta delta}
    Let $H$ be a closed subgroup of $G$ with finitely many connected components. If $\thetagh> 1/2$, then $\betagh=\thetagh$.
\end{myproposition}
\begin{remark*}
    (1) This statement is sharp in the sense that it fails for values below $1/2$ without further assumptions, as we explain in \autoref{eg: delta ne beta for G/N}.

    (2) Combining \autoref{mythm: theta and delta} and \autoref{mythm: rho theta delta}, one sees that for closed subgroups of $G$ with finitely many components, the two geometric exponents $\deltagh$ and $\betagh$ are equal if $\deltagh>1/2$.
\end{remark*}

Our proof uses unitary representation theory which only gives information above $1/2$ (i.e.\,when $\Lspace^2(G/H)$ is non-tempered). We apply the same strategy as \autoref{mythm: theta and delta}, along with the extra input of spherical functions and ingredients borrowed from \cite{benoist-kobayashi2022tempered2}.

Finally, we present the following result of independent interest, to be compared with the existing statements on discrete subgroups in \cite{roblin2005un-theoreme,coulon-dougall2025twisted,glorieux-tapie2020critical}.
\begin{myproposition}[\autoref{lem: coamenable subgps have same exponent}]
    For any closed subgroups $F<H<G$ with $\Lspace^2(G/H)$ non-tempered (or equiv.\,$\deltagh>1/2$), if $F$ is co-amenable in $H$, then $\delta_{G/F}=\delta_{G/H}$.
\end{myproposition}

\subsection{Organization}
In \autoref{sec Lie gp}, we recall some elements in the theory of semisimple Lie groups. In \autoref{sec unitary rep}, we recall some definitions and facts about unitary representations. In \autoref{sec vol growth decay}, we establish the fundamental tools to address the growth and decay of volume in real semisimple groups, which provide the machinery for the main proofs. In \autoref{sec decay and vol growth}, we prove \autoref{mythm: theta and delta} and then deduce \autoref{mythm: reductive subgp, delta=theta=rho} and \autoref{mythm: discrete subgp, delta=theta=}. In \autoref{sec unif decay}, we prove \autoref{mythm: rho theta delta}.

% section
\section{Analysis on semisimple Lie groups}\label{sec Lie gp}
\subsection{Measures on homogeneous spaces}\label{subsec: measures on homogeneous spaces}
Let $G$ be a locally compact group and $\dif x$ be a left Haar measure on $G$. 
The convention for the modular function $\Delta_G: G\to \bbR_{>0}$ depends on the author. Here we define it as the continuous group morphism such that
\begin{equation*}
    \int_G f(xg^{-1})\dif x = \Delta_G(g) \int_G f(x)\dif x, \;\;\textup{for all } f\in \Cc(G),
\end{equation*}
or formally $ \dif(xg) = \Delta_G(g) \dif x$.
If $G$ is a Lie group and $\Ad_G: G\to \GL(\glie)$ is the adjoint representation, then we have $\Delta_G(g) = \det\Ad_{G}(g)^{-1}$.

A locally compact group $G$ is unimodular if $\Delta_G\equiv 1$.
In general, a right Haar measure on $G$ can be defined by
$ \dif(x^{-1}) = \Delta_G(x)^{-1} \dif x$.
What will play a role later is the \emph{symmetric measure} $\nu_G$ on $G$ defined by
\begin{equation}\label{eq def: symmetric measure}
    \dif \nu_G(x) := \Delta_G(x)^{-\frac{1}{2}}\dif x.
\end{equation}
The symmetry comes from the observation that $\dif \nug{x^{-1}} = \dif \nu_G(x)$.

Let $H$ be a closed subgroup of the locally compact group $G$. The quasi-invariant measures on the homogeneous space $G/H$ are characterized by the following lemma.
Be noted that the integration formula holds up to normalization of the Haar measures. On locally compact groups, we take the left Haar measures unless otherwise stated.
\begin{lemma}[{\cite[Lem B.1.3]{bekka-harpe2008kazhdans}}]\label{lem: delta function and quasi-invariant measure on homogeneous space}
    The homogeneous space $G/H$ always admits a $G$-quasi-invariant Radon measure. More precisely, the following data are equivalent:
    \begin{enumerate}
        \item a density function $\delta: G\to \bbR_{>0}$ which is continuous and satisfies
        \begin{equation}\label{eq ppty: delta-function, measures on homogeneous spaces}
            \delta(gh) = \frac{\Delta_H(h)}{\Delta_G(h)} \delta(g)
        \end{equation}
        for all $g\in G$ and $h\in H$;
        \item a quasi-invariant Radon measure $\mu$ on $G/H$.
    \end{enumerate}
    The connection between these two items is given by
    \begin{equation}\label{eq ppty: integration on homogeneous space}
        \int_G f(g)\delta(g) \dif g = \int_{G/H} \int_H f(gh)\dif h \dif\mu(gH)
    \end{equation}
    for all $f\in\Cc(G)$. Moreover, the Radon--Nikodym derivative is given explicitly for all $g\in G$ and $x\in G$ by
    \begin{equation*}
        \frac{\drm(g_*\mu)}{\drm \mu}(xH) = \frac{\delta(g^{-1}x)}{\delta(x)}.\qedineq
    \end{equation*}
\end{lemma}
In particular, when there exists a $G$-invariant Radon measure on $G/H$, i.e.\,when $\Delta_G|_H \equiv \Delta_H$, such a measure is unique up to scalar.
The following lemma serves as a tool to deduce integration formulae on Lie groups, while being general itself.

\begin{lemma}[{\cite[Prop 5.26]{knapp1986representation}}]\label{lem: haar measure from product of two closed subgroups}
    Let $S,T$ be closed subgroups of $G$ so that the complement of $ST$ in $G$ has zero Haar measure, while $K=S\cap T$ is a compact subgroup. Then we can normalize the Haar measures so that for all $f\in\Cc(G)$ we have
    \begin{equation*}
        \int_G f(g) \dif g = \int_S\int_T f(st)\frac{\Delta_G(t)}{\Delta_T(t)} \dif s \dif t.
        \qedineq
    \end{equation*}
\end{lemma}

\subsection{Semisimple groups and parabolic subgroups}
The general references for this part include \cite{knapp1986representation,helgason2001differential}.

For the rest of this section, let $G$ denote a noncompact semisimple real algebraic group, but the results to be discussed are valid for real reductive groups with mild modifications.
The rich structure theory of $G$ gives rise to a variety of integration formulae.
We continue with the notation introduced in \autoref{subsec: notation and reformulation}.
The semisimple Lie algebra $\glie$ admits an adjoint-invariant inner product $B_\theta(X,Y) = -B_0(X,\theta Y)$,
where $B_0(X,Y):=\tr(\ad X \ad Y)$ is the Killing form of $\glie$.

Let $\mlie := Z_{\klie}(\alie)$ be the centralizer of $\alie$ in $\klie$.  In the root space decomposition, we have $\glie_0 = \mlie\oplus\alie$.
Let $M$ and $M'$ be respectively the centralizer and normalizer group of $A$ in $K$. Then $\Lie(M)=\Lie(M')=\mlie$ and the finite quotient group $M'/M = W(\alie;\glie)=:W_G$ is the (restricted) Weyl group which acts simply transitively on the set of Weyl chambers of $\alie$.

Let $\nlie=\oplus_{\alpha\in\Sigma^+}\glie_\alpha$ and $N$ be the associated analytic subgroup of $G$. Then we have the Iwasawa decomposition $G=KAN$ where each component is uniquely defined.
Denote the Iwasawa projection by 
$$\eta: G\to \alie,\;\; g \in Ke^{\eta(g)}N.$$

\subsubsection{The Cartan projection}
The inner product $B_\theta$ allows one to identify $\alie$ with the dual space $\alie'$. We write $\norm{\cdot}$ for the corresponding Euclidean norm on $\alie$. This norm is $W_G$-invariant by the ad-invariance of $B_\theta$. Denote by $\alie(r)$ the closed metric ball centered at $0$ of radius $r$.

Recall that $\kappa: G\to \alie^+$ denotes the Cartan projection.
We say that a sequence $(g_n)_{n\in\bbN}$ of elements in $G$ \emph{go to infinity} (write $g_n\to \infty$), if they eventually leave every compactum of $G$. This is equivalent to saying that $\norm{\kappa(g_n)}\to+\infty$ as $n\to\infty$. This notion is compatible with group translation on both left and right, as is shown by the following lemma.

\begin{lemma}[{\cite[Prop 5.1]{benoist1996actions}}]\label{lem: compact error in cartan projection}
    For any compact subset $B$ of $G$, there exists $r=r(B)>0$ such that $\kappa(BgB)\subset \kappa(g) + \alie(r)$ for all $g\in G$.\qed
\end{lemma}

\subsubsection{Parabolic subgroups}\label{subsec: parabolic subgroups}
    The closed subgroup $Q_0:=MAN$ is a minimal parabolic subgroup of $G$.
    Let $Q$ be a parabolic subgroup of $G$ containing $Q_0$ and $Q=M_Q A_Q N_Q$ be the Langlands decomposition of $Q$, cf.\,\cite[\S V.5]{knapp1986representation}. The subgroup $M_Q$ is reductive and the subgroup $L:=M_Q A_Q$ is a Levi factor of the parabolic subgroup $Q$ with $Q=L\ltimes N_Q$.

    Let $\Pi'\subset \Pi$ be the subset of simple positive roots orthogonal to the subspace $\alie_Q$.
    Write $\langle\Pi'\rangle$ for the span of $\Pi'$. Dually, $\alieQ$ is the orthogonal complement of $\langle \Pi'\rangle$. Define 
    $\Sigma_Q^+ := \Sigma^+\setminus\langle\Pi'\rangle$ and $\Sigma_M^+:=\Sigma^+\cap \langle\Pi'\rangle$.
    With the Euclidean structure on $\alie$ induced by the Killing form, let
    $$ \alieM:=\alieQ^{\perp} \textup{ in }\alie, \;\; \nlieM:= \bigoplus_{\alpha\in \Sigma_M^+} \glie_\alpha.$$
    Then as vector spaces, we have
    \begin{align*}
        \mlieQ = \mlie\oplus \alieM \oplus \nlieM\oplus \theta\nlieM, \;\; \nlieQ = \bigoplus_{\alpha\in \Sigma_Q^+}\glie_\alpha,\;\; \alie=\alieM\oplus\alieQ,\;\;\nlie=\nlieM\oplus\nlieQ.
    \end{align*}
    Let $K_M=K\cap M_Q$ and $A_M, N_M$ be the analytic subgroups corresponding to $\alieM,\nlieM$. Then $ M_Q =K_M A_M N_M$ is an Iwasawa decomposition of $M_Q$, $A=A_M A_Q \cong A_M \times A_Q$, and $N=N_M N_Q \cong N_M \ltimes N_Q$.
    The group $M_Q$ centralizes $A_Q$ and normalizes $N_Q$.
    We remark that all the groups discussed here are closed subgroups of $G$.

    \begin{notation*}
        Define the following real linear forms on $\alie$:
        \begin{align*}
            \rhoQ = \frac{1}{2}\sum_{\alpha\in \Sigma_Q^+} (\dim \glie_\alpha) \alpha,\;\;
            \rhoM = \frac{1}{2} \sum_{\alpha\in \Sigma_M^+} (\dim \glie_\alpha) \alpha,
        \end{align*}
        so that $\rho=\rhoQ+\rhoM$.
        Sometimes $\rho_M$ is denoted by $\rho_L$.
        For $X\in\alie$, write $X_Q,X_M$ respectively for the orthogonal projection of $X$ to the subspaces $\alie_Q,\alie_M$.
    \end{notation*}
    Recall that $\eta:G\to\alie$ denotes the Iwasawa projection.
    The modular function of a parabolic subgroup $Q$ can be computed from the adjoint action: we have $\Delta_Q(q)  = e^{-2\rhoQ\eta(q)}$ and in particular $\Delta_{Q_0}(q)=e^{-2\rho\eta(q)}$.
    Remark that the symmetric measure on $Q$ is thus given from \eqref{eq def: symmetric measure} by $\dif\nu_Q(q)  = e^{\rhoQ\eta(q)} \dif q$.

    \begin{lemma}\label{lem: rhom=0 on aq, rhoq=0 on am}
        For every $X\in\alie$, we have $\rhoQ(X_M)=\rhoM(X_Q)=0$.
    \end{lemma}
    \begin{proof}
        That $\rhoM(X_Q)=0$ follows directly from the definition. 
        To prove $\rhoQ(X_M)=0$, let us assume that $\Pi'\ne\emptyset$; otherwise, there is nothing to prove.
        Dually, this is equivalent to $\alpha \perp \rhoQ$ for all $\alpha\in \Pi'$. But any $\alpha\in\Pi'$ is a simple root, so the $\alpha$-reflection $s_\alpha$ preserves setwise $\Sigma\cap\spanangle{\Pi'}$ and hence also $\Sigma^+\setminus\spanangle{\Pi'}$. But that means $s_\alpha(\rhoQ)=\rhoQ$, i.e.\,$\alpha\perp\rhoQ$.
    \end{proof}

    \begin{lemma}\label{cor: rhoQeta is KM right invariant}
        The map $\rhoQ\circ\eta: G\to \bbR$ is $K_M$-right-invariant.
    \end{lemma}
    \begin{proof}
        Since the group $M_Q$ normalizes both $A_Q$ and $N_Q$, \autoref{lem: rhom=0 on aq, rhoq=0 on am} implies $\rhoQ\eta(kman) = \rhoQ(\log a)$ for any $k\in K$, $m\in M_Q$, $a\in A_Q$, and $n\in N_Q$.
        For $k_M\in K_M$, we have $kman k_M = kmk_M a n'$ where $n'=k_M^{-1}n k_M\in N_Q$. Hence, we have $\rhoQ\eta(kmank_M)=\rhoQ\eta(kman)$.
    \end{proof}

\subsubsection{Integral formulae}
\autoref{lem: haar measure from product of two closed subgroups} implies the following formula.

\begin{lemma}[Integration from Iwasawa decomposition]\label{prop: KQ integral, Q psg}
    Let $Q$ be a parabolic subgroup as before.
        For any integrable function $f$ on $G$, we have
        \begin{equation*}
            \int_G f(g) \dif g = \int_K \int_Q f(kq)\, e^{2\rhoQ\eta(q)} \dif k \dif q.
            \qedineq
        \end{equation*}
\end{lemma}

\begin{lemma}[Integration from Cartan decomposition]\label{prop: KAK Cartan integral}
    For any integrable function $f$ on $G$, we have
    \begin{equation*}
        \int_G f(g) \dif g = \int_K \int_{\alie^+} \int_K f(k_1 e^X k_2) \left(\prod_{\alpha\in\Sigma^+}\sinh^{\dim \glie_\alpha} \alpha(X)\right) \dif k_1\dif X\dif k_2.
        \qedineq
    \end{equation*}
\end{lemma}

For each group element $w\in W_G$, fix a representative $m_w\in M'$ and write $N^w:=m_w N m_w^{-1}$ which is independent of the choice of $m_w$. Let $w^*$ be the unique element of $W$ which maps $\alie^+$ to $-\alie^+$. Denote by $\Nbar:=N^{w^*}=\Theta N$.
\begin{lemma}[Bruhat decomposition]\label{thm: Bruhat decomposition}
    We have the following decomposition
    \begin{equation*}\label{eq res: Bruhat decomposition}
        G=\bigsqcup_{w\in W_G} Q_0 m_w Q_0 = \bigsqcup_{w\in W_G}MANN^w m_w,
    \end{equation*}
    where the double class $Q_0 m_{w^*}Q_0 = m_{w^*}\Nbar MAN$ is an open submanifold of $G$ of full Haar measure, while the other double classes are submanifolds of strictly lower dimensions.
    \qed
\end{lemma}

Hence, $\Nbar MAN$ is an open submanifold of $G$ whose complement has zero Lebesgue measure. Moreover, multiplication map $\Nbar\times MAN\to \Nbar MAN$ is a diffeomorphism. We also have the following formula by \autoref{lem: haar measure from product of two closed subgroups}.

\begin{lemma}[Integration from Bruhat decomposition]\label{prop: integration via Bruhat decomposition}
    For any integrable function $f$ on $G$, we have
    \begin{equation*}
        \int_G f(g)\dif g = \int_{\Nbar}\int_M\int_A\int_N f(\nbar man) \,e^{2\rho\log a} \dif\nbar \dif m \dif a \dif n.
        \qedineq
    \end{equation*}
\end{lemma}

\subsection{Spherical functions}\label{subsec: spherical fn}
To each real linear form $\chi\in \alie'$, we associate the following function on $G$:
\begin{equation*}
    \Xi^G_{\chi}(g) := \int_{K} e^{-(\chi+\rho)\left(\eta(g^{-1}k)\right)}\dif k,
\end{equation*}
where $\eta: G\to \alie$ is the Iwasawa projection.
These smooth $K$-bi-invariant functions are \emph{spherical functions} for the pair $(G,K)$. 
We have the following invariance property of the parameter $\chi$.

\begin{lemma}[{\cite[Prop 7.15]{knapp1986representation}}]\label{lem: W_G invariance, spherical functions}
    For any $\chi\in\alie'$ and $w\in W_G$, we have $\Xi_{\chi}^G = \Xi^G_{w\chi}$.\qed
\end{lemma}

Under the natural identification between the Cartan subspace $\alie$ and $\alie'$, let the positive Weyl chamber $\alie^+$ correspond to $(\alie')^+$.
For those $\chi$ in the interior of the convex hull of the $W$-orbit of $\rho$, the spherical functions decay exponentially fast. In fact, their precise asymptotics can be determined, cf.\,\cite[Thm 3.4]{narayanan-pasquale2014asymptotics}.
We will only need the following estimates which are more classical, cf.\,\cite[\S VII.8-9]{knapp1986representation} and \cite[Exer IV.B.1]{helgason2000groups}.

\begin{lemma}\label{lem: spherical functions}
    For each $\chi\in (\alie')^+$, there exists a polynomial $p(\cdot)$ on $\alie$ such that for all $g\in G$, we have
    \begin{equation*}
        \bigexp{(\chi-\rho)\kappa(g)} \leq \Xi^G_{\chi}(g) \leq p(\kappa(g))\bigexp{(\chi-\rho)\kappa(g)}.
        \qedineq
    \end{equation*}
\end{lemma}

% section
\section{Unitary representations}\label{sec unitary rep}

The general references include \cite{knapp1986representation,bekka-harpe2008kazhdans}.

Let $G$ be a locally compact group. A \emph{unitary representation} $(\pi,\Hil)$ of $G$ consists of $\Hil$ a complex Hilbert space and $\pi: G\to \Ulie(\Hil)$ a group morphism from $G$ to the group $\Ulie(\Hil)$ of unitary operators on $\Hil$ which is \emph{strongly continuous} in that for any $v\in \Hil$, the map $G\to \Hil$, $g\mapsto \pi(g)v$ is continuous.
A \emph{matrix coefficient} of $\pi$ is a map of the form 
$$ G\to \bbC,\;\; g\mapsto \Innerprod{\pi(g)v_1, v_2},$$
where $v_1,v_2\in \Hil$. By strong continuity, matrix coefficients are bounded continuous functions on $G$.

Two unitary representations $(\pi_1,\Hil_1)$ and $(\pi_2,\Hil_2)$ are \emph{unitarily equivalent} if there exists a $G$-intertwining unitary isomorphism between $\Hil_1$ and $\Hil_2$, in which case $\pi_1,\pi_2$ have the same collection of matrix coefficients.

\begin{example*}
    Let $\dif x$ be a left Haar measure on $G$. The convention for the $\Lspace^2$-scalar product of functions is
    $$ \Innerprod{f_1, f_2} = \int_G f_1(x)\overline{f_2(x)} \dif x.$$
    The \emph{left regular representation} $\lambda_G$ is the unitary representation of $G$ on the Hilbert space $\Lspace^2(G)$ acting by
    $$ \lambda_G(g)f : x \mapsto f(g^{-1}x),\;\; \textup{for } f\in \Lspace^2(G).$$
\end{example*}

\subsection{Induced representations}\label{subsec: induced representations}
Let $G$ be a locally compact group, $H$ be a closed subgroup of $G$, and $(\sigma,\calV)$ be a unitary representation of $H$.
Let $\mu$ be a quasi-invariant Radon measure on the homogeneous space $G/H$ and $\delta$ be the associated function satisfying \eqref{eq ppty: delta-function, measures on homogeneous spaces} (cf. \autoref{lem: delta function and quasi-invariant measure on homogeneous space}).

We describe the induced unitary representation $(\pi,\Hil):= \Ind_H^G(\sigma,\calV)$. Elements of $\Hil$ are measurable vector-valued functions $f: G\to \calV$ with $\sigma$-equivariance $f(xh) = \sigma(h)^{-1} f(x)$
for all $x\in G$ and $h\in H$,
and the $\Lspace^2$-integrability
$$ \|f\|^2 := \int_{G/H} \Innerv{f(x),f(x)} \dif\mu(xH) < +\infty,$$
where $\Innerv{f(x),f(x)}$ does not depend on the representative of $xH$ since $\sigma$ is unitary.
The induced action of $G$ is given by
\begin{align*}
    \pi(g)f(x) = f(g^{-1}x) \left(\frac{\drm g_*\mu}{\drm\mu}(xH)\right)^{\frac{1}{2}} = f(g^{-1}x) \left(\frac{\delta(g^{-1}x)}{\delta(x)}\right)^{\frac{1}{2}}.
\end{align*}
Here, the cocycle term ensures that $\pi(g)$ is a unitary operator. Although \emph{a priori} this definition depends on the measure $\mu$, it turns out that different choices of $\mu$ give unitarily equivalent representations.
In particular, if $\sigma=1_H$, then $\pi$ is the \emph{quasi-regular representation} $\lambda_{G/H}$ on $\Lspace^2(G/H)$.

There is a simple way to produce elements in $\Hil$. For $\varphi\in \Cc(G)$ and $v\in \calV$, define the map $\rmI^G_H(\varphi, v) = \rmI(\varphi,v): G\to \calV$ by
\begin{equation}\label{eq def: I(phi, v)}
    \rmI(\varphi,v)(x):= \int_H \varphi(xh)\sigma(h)v \dif h.
\end{equation}
We only specify $\rmI_H^G$ when necessary.
The equivariance property follows from
$$ \rmI(\varphi,v)(xh_0) = \int_H \varphi(xh_0 h)\sigma(h)v \dif h = \sigma(h_0)^{-1} \rmI(\varphi,v)(x)$$
for all $x\in G$ and $h_0\in H$,
and the $\Lspace^2$-integrability from
\begin{align*}
    \| \rmI(\varphi,v)\|^2 &= \int_{G/H} \Normv{\int_H \varphi(xh) \sigma(h)v\dif h}^2 \dif \mu(xH)\\
    &\leq \Normv{v}^2 \int_{G/H} \int_H \norm{\varphi(xh)}^2 \dif h\dif\mu(xH)\\
    &= \Normv{v}^2 \int_G |\varphi(g)|^2 \delta(g) \dif g < +\infty.
\end{align*}
Hence, the function $\rmI(\varphi,v)$ belongs to $\Hil$. We have the following fact.

\begin{lemma}[{\cite[Lem B.1.2]{bekka-harpe2008kazhdans}}]\label{lem: surjectivity of Cc(G) to Cc(GH)}
    Let $\sigma=1_H$. Then the map $\Cc(G)\to \Cc(G/H)$ given by $\varphi\mapsto \rmI(\varphi, 1)$ is surjective.
    \qed
\end{lemma}

\begin{lemma}\label{lem: decompose each CcG into small ones}
    Given any neighborhood $B_G$ of $e$ in $G$ and any $\psi\in \Cc(G)$, there exists finitely many $\varphi_i\in \Cc(G)$ with $(\supp\varphi_i)(\supp\varphi_i)^{-1}\subset B_G$ such that
    $ \psi = \sum_{i} \varphi_i$.
\end{lemma}
\begin{proof}
    Since $B_G$ is a neighborhood of $e$, we can find a relatively compact open neighborhood $B$ of $e$ with $BB^{-1}\subset B_G$. Now
    $ \setdef{Bg: g\in \supp\psi} $
    gives an open cover of $\supp\psi$. By the compactness of $\supp\psi$, there exists a finite subcover $\setdef{Bg_i: i\in I}$ for some finite index set $I$. Then we can find a finite partition of unity $\setdef{\chi_i\in \Cc(Bg_i)}_{i\in I}$ subordinated to this cover, so that $1 = \sum_{i\in I} \chi_i$ over $\supp\psi$.
    Then $\varphi_i:= \psi\chi_i\in \Cc(G)$ satisfies $(\supp\varphi_i)(\supp\varphi_i)^{-1}\subset B_G$ and 
    $ \psi = \sum_{i} \varphi_i$.
\end{proof}

\begin{corollary}\label{cor: I(varphi,1) spans Cc(GH)}
    Let $\sigma=1_H$.
    For any neighborhood $B_G$ of $e$ in $G$, the set 
    $$\setdef{\rmI(\varphi,1): \varphi\in\Cc(G) \textup{ with } (\supp\varphi)(\supp\varphi)^{-1}\subset B_G}$$
    spans $\Cc(G/H)$ and hence is a total subset in $\Ltwogh$.
\end{corollary}
\begin{proof}
    The corollary follows from \autoref{lem: surjectivity of Cc(G) to Cc(GH)} and \autoref{lem: decompose each CcG into small ones}.
\end{proof}

\begin{lemma}[Induction by stage {\cite[Thm E.2.4]{bekka-harpe2008kazhdans}}]\label{lem: induction by stage}
    Let $F<H<G$ be two closed subgroups of $G$. Then for any unitary representation $\sigma$ of $F$, we have $\Ind_F^G\sigma \cong \Ind_{H}^G(\Ind_F^H \sigma)$.\qed
\end{lemma}

\subsection{Weak containment}\label{subsec: weak containment}
We introduce the notion of weak containment to study unitary representations of noncompact groups.
Let $G$ be a lcsc group.
\begin{definition}
    Let $(\sigma,\calV)$ and $(\pi, \Hil)$ be two unitary representations of $G$.
    Say that $\sigma$ is \emph{weakly contained} in $\pi$ (write $\sigma\prec\pi$), if every diagonal matrix coefficient $\Innerprod{\sigma(\cdot)v,v}$ can be approximated, uniformly on compacta, by convex combinations of diagonal matrix coefficients of $\pi$.
\end{definition}

\begin{fact}\label{fact on weak containment}
    (1) $\sigma\prec \pi$ iff $\opnorm{\sigma(f)}\leq \opnorm{\pi(f)}$ for all $f\in \Cc(G)$.
    
    (2) Weak containment is preserved under induction and restriction of unitary representations.
    
    (3) A locally compact group is amenable iff the trivial representation is weakly contained in its regular representation.
\end{fact}

\begin{example}\label{eq: L2(G/P)}
    Let $G$ be a real semisimple algebraic group and $Q_0$ be a minimal parabolic group. As $Q_0$ is amenable, the trivial representation $1_{Q_0}$ is weakly contained in the regular representation $\lambda_{Q_0}$. Hence, $\lambda_{G/Q_0} = \Ind_{Q_0}^G 1_{Q_0}$ is weakly contained in $\lambda_G=\Ind_{Q_0}^G \lambda_{Q_0}$.
\end{example}

\subsection{Tempered representations}
Now let $G$ be a noncompact real semisimple algebraic group. 
\begin{definition}
    A unitary representation $\pi$ of $G$ is \emph{tempered} if $\pi$ is weakly contained in the regular representation $(\lambda_G, \Lspace^2(G))$.
\end{definition}

\autoref{eq: L2(G/P)} implies that $(\lambda_{G/Q_0}, \Lspace^2(G/Q_0))$ is tempered.
Note that the matrix coefficient $\Innerprod{\lambda_{G/Q_0}(\cdot)1,1}$ equals the Harish-Chandra spherical function $\Xi_0$ (\autoref{subsec: spherical fn}). In general, the matrix coefficients of tempered representations are characterized as follows.

\begin{theorem}[{\cite{cowling-haagerup1988almost}}]\label{thm: temperedness, equivalent formulations}
    Let $G$ be a real semisimple algebraic group and $K$ be a maximal compact subgroup of $G$.
    Then for any unitary representation $(\pi, \Hil)$ of $G$, the following statements are equivalent:
    \begin{enumerate}[label=(\roman*)]
        \item $\pi$ is tempered;
        \item for all $K$-finite vectors $v_1, v_2$ in $\Hil$, we have
        \begin{equation*}
            \Modofinner{\pi(g) v_1, v_2} \leq \sqrt{\dim\spanangle{Kv_1} \dim\spanangle{Kv_2}} \Norm{v_1} \Norm{v_2} \Xi^G_0(g);
        \end{equation*}
        \item there exists a dense subspace $\Hil_0$ of $\Hil$, such that for all $v\in\Hil_0$, the coefficients $\Innerprod{\pi(\cdot)v,v}\in \Lspace^{2+\varepsilon}(G)$ for any $\varepsilon>0$.\qed
    \end{enumerate}
\end{theorem}

In view of \autoref{lem: spherical functions}, the optimal decay of spherical functions is given by $\Xi^G_0$.
Meanwhile, it follows from \autoref{lem: spherical functions} and \autoref{prop: KAK Cartan integral} that we have $\Xi^G_0\in \Lspace^{2+\varepsilon}(G)$ for any $\varepsilon>0$.
For a closed subgroup $H$ of $G$, we have the following consequence.
\begin{corollary}\label{cor: temperedness iff thetagh}
    $\Ltwogh$ is tempered iff $\thetagh\leq 1/2$. \qed
\end{corollary}

% section
\section{Volume growth and volume decay}\label{sec vol growth decay}
The goal of this section is to establish the preliminary tools to study the growth and decay of volume in semisimple Lie groups.

Let $G$ be a real semisimple algebraic group and $\nu_G$ be its Haar measure.
We will be interested for example in the behavior of the volume of $BgB$, where $B$ is a fixed compactum while $g$ varies in $G$. The consideration of such volume functions is classical and appears in the related work \cite{ganguly-krotz2025a-note}.

\subsection{Local volume decay in $G$}
As a consequence of the Bruhat decomposition (\autoref{thm: Bruhat decomposition}), the multiplication map $\Nbar\times M\times A\times N\to G$ is a diffeomorphism onto an open subset of full measure.
\begin{lemma}\label{lem: volume decay in NMAN}
    For any compact subset $B$ of $\Nbar MAN$, there exists a constant $C>0$ such that for all $a\in A^+$ we have $\nug{a B a^{-1} \cap B} \leq C e^{-2\rho \log a}.$
\end{lemma}
\begin{proof}
    By the Bruhat decomposition, there exist compact sets 
    $$B_{\Nbar}\subset \Nbar,\; B_A\subset A,\; B_N\subset N$$
    such that $B \subset B_{\Nbar}M B_A B_N$. From \autoref{prop: integration via Bruhat decomposition} we deduce
    \begin{align*}
        \nug{a B a^{-1} \cap B}\leq \int_{B_\Nbar}\int_M\int_{B_A}\int_{B_N}  \indicator_{B}(a^{-1}\nbar ma_1 n a)\,e^{2\rho\log a_1} \dif\nbar \dif m\dif a_1 \dif n.
    \end{align*}
    Since $a$ normalizes both $\Nbar$ and $N$, and since $MA$ centralizes $a$, we have furthermore
    \begin{align*}
        \nug{a B a^{-1} \cap B} &\leq \int_{B_{\Nbar}\cap a B_\Nbar a^{-1}}\dif \nbar \int_M \dif m \int_{B_A} e^{2\rho\log a_1} \dif a_1 \int_{B_{N}}\dif n\\
        &\leq \int_{a^{-1}B_\Nbar a\cap B_\Nbar} e^{-2\rho \log a}\dif \nbar' \,C(B_A, B_N)\\
        &\leq C(B_\Nbar, B_A, B_N) \,e^{-2\rho\log a},
    \end{align*}
    where we set $\nbar'=a^{-1}\nbar a$ and get $\dif\nbar' = e^{2\rho\log a} \dif \nbar$.
\end{proof}

\begin{lemma}\label{lem: open nbhd B_G, shrink NMAN}
    There exists an open neighborhood $B_G$ of $e$ in $G$, such that for all $k\in K$ we have $kB_G k^{-1}\subset \Nbar MAN$.
\end{lemma}
\begin{proof}
    This follows from the openness of $\Nbar MAN$ in $G$ (\autoref{thm: Bruhat decomposition}) and the compactness of $K$.
\end{proof}

\begin{proposition}\label{prop: local volume decay in G}
    Let $B_G$ be given as in \autoref{lem: open nbhd B_G, shrink NMAN}. Then for any functions $\varphi_1,\varphi_2\in \Cc(B_G)$, there exists a constant $C=C(\varphi_1,\varphi_2)$ such that uniformly for all $g\in G$,
    \begin{equation*}
        \left|\int_G \varphi_1(g^{-1}x g)\varphi_2(x)\dif x\right| \leq C e^{-2\rho\kappa(g)}.
    \end{equation*}
\end{proposition}
\begin{proof}
    For $i=1,2$, define $\tilde{\varphi}_i(x) := \sup_{k\in K} \left|\varphi_i(kxk^{-1})\right|$.
    % \begin{equation*}
    %     \tilde{\varphi}_i(x) := \sup_{k\in K} \left|\varphi_i(kxk^{-1})\right|.
    % \end{equation*}
    Then $\tilde{\varphi}_1, \tilde{\varphi}_2$ are continuous functions compactly supported in $\Nbar MAN$ by \autoref{lem: open nbhd B_G, shrink NMAN}.
    For $g\in G$, we can write $g=k_2 e^X k_1$ for $X=\kappa(g)$ and some $k_1,k_2\in K$. By the unimodularity of $G$, we have
    \begin{align*}
        \norm{\int_G \varphi_1(g^{-1}x g)\varphi_2(x)\dif x}
        &= \norm{\int_G \varphi_1(k_1^{-1} e^{-X}y e^X k_1 )\varphi_2(k_2 y k_2^{-1})\dif y}\\
        &\leq \int_G \tilde{\varphi}_1(e^{-X} y e^X) \tilde{\varphi}_2(y) \dif y,
    \end{align*}
    which is bounded from above by $C e^{-2\rho(X)}$ by \autoref{lem: volume decay in NMAN}.
\end{proof}

\subsection{The rho-function and volume decay}\label{subsec: the rho-function and vol decay}
Let $H$ be a Lie group and $R: H \to \GL(V)$ be a continuous linear representation on a $d$-dimensional real vector space $V$.
By notation abusing, the differential map $R: \hlie \to \End(V)$ is a representation of the Lie algebra $\hlie$.
To these data we associate the following rho-function $\rho_V: \hlie\to \bbR^+$.
\begin{definition}
    For each $Y\in\hlie$, the action of $R(Y)$ on $V\otimes \bbC$ admits a Jordan normal form over $\bbC$ with diagonal elements $\lambda_1,\dots,\lambda_d$. We define
    \begin{equation*}
        \rho_V(Y) := \frac{1}{2} \sum_{i=1}^d \norm{\Re\lambda_i}.
    \end{equation*}
\end{definition}

It follows from the definition that $\rho_V$ is a continuous homogeneous function which is invariant by the adjoint action of $H$.

\begin{remark}\label{rem: rho function for alg gp}
    Let $H$ be a real algebraic group and $R$ be an algebraic representation over $\bbR$.
    Let $\alieh$ be a maximal split abelian subalgebra of $\hlie$. Since $R(\alieh)$ is a split abelian subalgebra of $\End(V)$, the action of $\alieh$ is jointly diagonalizable over $\bbR$. Then the restriction $\rho_V|_{\alieh}$ is the half sum of the absolute values of the eigenvalues and therefore is a piecewise linear, continuous, convex, homogeneous function. As $V$ is finite-dimensional, the function $\rho_V$ is uniformly Lipschitz on $\alieh$.
    If $R$ is faithful, then $\rho_V|_{\alieh}$ is a polyhedral norm on $\alieh$.

    By the Jordan decomposition, every element $Y\in\hlie$ splits uniquely as a sum of commuting elements $Y=Y_e + Y_h + Y_n$ in $\hlie$ with $Y_e$ elliptic, $Y_h$ hyperbolic and $Y_n$ nilpotent. Moreover, $Y_h$ is $H$-conjugated into $\alieh$. Since $\rho_V(Y)=\rho_V(Y_h)$, the function $\rho_V$ is determined by $\rho_V|_{\alieh}$.
\end{remark}

\begin{example}\label{eg: rhoh in case of adjoint rep}
    Let $H$ be a real reductive group and $(R, V)=(\Ad, \hlie)$ be the adjoint representation. Fix a positive system $\Sigma^+(\alieh; \hlie)$ and let $\rhoH$ be the usual \emph{half} sum of positive roots.
    Then the convex function $\rhoh$ coincides with the twice of the linear form $\rhoH$ on the positive Weyl chamber $\alieh^+$. If $W_H=W(\alieh;\hlie)$ denotes the Weyl group, then for all $X\in\alieh$,
    $$\rhoh(X) = \max_{w\in W_H} 2\rhoH(wX).$$
    In particular, $\rhoh$ is $W_H$-invariant.
\end{example}

Let $\Vol$ be the Lebesgue measure on the vector space $V\cong \bbR^d$. The function $\rho_V$ reflects the volume decay of the representation $R$.
\begin{lemma}[{\cite[Lem 2.8]{benoist-kobayashi2022tempered2}}]\label{lem: rho fn and volume decay}
    Let $\alie\subset \End(V)$ be a real split abelian subalgebra.
    Then for any compact neighborhood $B$ of $0$ in $V$, there exist constants $c, C>0$ such that uniformly for all $X\in\alie$,
    \begin{equation*}
        c e^{-\rho_V(X)} \leq e^{-\Tr R(X)/2}\Vol\!\left(R(\exp X) B\cap B\right) \leq C e^{-\rho_V(X)}.\qedineq
    \end{equation*}
\end{lemma}

\begin{corollary}\label{cor: rho fun semisimple and volume decay}
    For any compact neighborhood $B$ of $0$ in $V$ and any semisimple operator $X\in \End(V)$, there exist constants $c_X, C_X >0$ such that uniformly for all $t\in \bbR^+$,
    \begin{equation*}
        c_X e^{-t\rho_V(X)} \leq e^{-\Tr(tX)/2}\Vol\!\left(\exp(tX) B\cap B\right) \leq C_X e^{-t\rho_V(X)}.
    \end{equation*}
\end{corollary}
\begin{proof}
    One applies \autoref{lem: rho fn and volume decay} and absorbs the elliptic part which corresponds to a rotation on $V$.
\end{proof}

\autoref{lem: rho fn and volume decay} can be rephrased in terms of unitary representations.
\begin{corollary}\label{cor: rho fun and volume decay}
    Let $H$ be a real reductive group and $R: H\to \GL(V)$ be an algebraic linear representation. Consider the unitary representation $(\tau, \Lspace^2(V))$ of the group $H$ derived from $R$, given by
    \begin{equation}\label{eq def: unit rep on L2 of an alg rep}
        \tau(h)f(v) = f(R(h)^{-1}v) \left(\det R(h)\right)^{-\frac{1}{2}},
    \end{equation}
    for $f\in \Lspace^2(V)$ and $v\in V$. Then for any compact neighborhood $B$ of $0$ in $V$, there exist constants $c, C>0$ such that uniformly for all $h\in H$,
    \begin{equation*}
        c e^{-\rho_V\kappa_H(h)} \leq \Innerprod{\tau(h)\indicator_{B}, \indicator_{B}} \leq C e^{-\rho_V\kappa_H(h)},
    \end{equation*}
     where $\kappa_H: H\to \alieh$ is a Cartan projection.
\end{corollary}
\begin{proof}
    Let $H=K_H A_H^+ K_H$ be the Cartan decomposition of $H$ associated with $\kappa_H$. The subset $D:=R(K_H)B$ is a $K_H$-left-invariant compact neighborhood of $0$ in $V$. Thus for any $k_1,k_2\in K_H$ and $X\in \alieh$, we have
    \begin{equation*}
        \Innerprod{\tau(k_1 e^X k_2)\indicator_D, \indicator_D} = e^{-\Tr R(X)/2} \Vol\!\left(R(e^X)D\cap D\right),
    \end{equation*}
    whence we can conclude by applying \autoref{lem: rho fn and volume decay}.
\end{proof}

Now we can define the local volume decay exponent of $G/H$.

\begin{definition}\label{def: betagh}
    Let $G$ be a semisimple Lie group and $H$ be a closed subgroup. The \emph{local volume decay exponent} $\betagh$ is defined by
    \begin{equation*}
        \betagh := \sup_{X\in\hlie} \frac{\rhoh(X)}{\rhog(X)},
    \end{equation*}
    where the spaces $\glie, \hlie$ are viewed as $\hlie$-module through the adjoint action. We take $0/0=0$ by convention.
\end{definition}
\begin{remark*}
    (1) By definition, the number $\betagh$ lies in $[0,1]$.

    (2) If $H_1<H_2$ are closed subgroups of $G$, then $\beta_{G/H_1}\leq \beta_{G/H_2}$.
    
    (3) If both $G$ and $H$ are algebraic groups, then \autoref{rem: rho function for alg gp} implies
    \begin{equation*}
        \betagh = \sup_{X\in\alieh} \frac{\rhoh(X)}{\rhog(X)}.
    \end{equation*}
\end{remark*}

\subsection{Volume growth in $G$}
By $B\Subset G$ we denote that $B$ be a compact subset of $G$ of nonempty interior. 
\begin{proposition}\label{prop: BgB volume growth in G}
    For any $B\Subset G$, there exist constants $c,C>0$ such that uniformly for all $g\in G$,
    \begin{equation*}
        c e^{2\rho\kappa(g)} \leq \nug{BgB} \leq Ce^{2\rho\kappa(g)}.\qedineq
    \end{equation*}
\end{proposition}

\begin{proof}
    Let $X = \kappa(g)$.
    For the upper bound,
    by \autoref{lem: compact error in cartan projection}, there exists $r>0$ such that $\kappa(BgB)\subset \kappa(g) + \alie(r)$ for all $g\in G$, whence we have $BgB \subset K e^{X+\alie(r)} K$. By \autoref{prop: KAK Cartan integral}, we have uniformly for $g\in G$,
    \begin{align*}
        \nug{BgB} \leq \int_K \int_{X+\alie(r)} \int_K e^{2\rho(Y)} \dif k \dif Y \dif k'
        \leq C \bigexp{2\rho(X)}.
    \end{align*}

    For the lower bound, first note that by \autoref{lem: compact error in cartan projection}, up to translation we can suppose that $B$ contains a neighborhood of $e \in G$. Then we can find a small neighborhood $B'$ of $e$ with $kB'k^{-1}\subset B$ for all $k\in K$ (cf.\,\autoref{lem: open nbhd B_G, shrink NMAN}).
    Write $g = k_1 e^X k_2$. By the unimodularity of $G$,
    $$\nug{BgB} = \nug{(k_1^{-1}Bk_1) e^X (k_2 B k_2^{-1})}\geq \nug{B'e^X B'}.$$
    Up to further shrinking $B'$, we can assume $B'\subset \Nbar  MAN$. 
    By the Bruhat decomposition (\autoref{thm: Bruhat decomposition}), there exist compact neighborhoods of $e$ in the respective subgroups $B_{\Nbar}\Subset \Nbar$, $B_M\Subset M$, $B_A\Subset A$, and $B_N\Subset N$ such that $ B_{\Nbar} B_M B_A B_N \subset B'$.
    Hence by the unimodularity of $G$, we deduce
    \begin{align*}
        \nug{B'e^X B'} &= \nug{B'e^X B' e^{-X}}\\
        & \geq \nug{B_{\Nbar}\left(e^X B_\Nbar B_M B_A B_N e^{-X}\right)}\\
        &= \nug{\left(B_{\Nbar}e^X B_{\Nbar} e^{-X}\right) B_M B_A \left(e^X B_N e^{-X}\right)}.
    \end{align*}
    By further applying \autoref{prop: integration via Bruhat decomposition}, we obtain uniformly for $g\in G$,
    \begin{align*}
        \nug{B'e^X B'} &\geq c\, \nunbar{B_{\Nbar}e^X B_{\Nbar} e^{-X}} \,\nun{e^X B_N e^{-X}}\\
        &\geq c\,\nunbar{B_\Nbar}\, e^{2\rho(X)} \nun{B_N} = c' e^{2\rho(X)}.\qedhere
    \end{align*}
\end{proof}

\subsection{Relative volume growth of closed subgroups}
In this subsection, we introduce the relative volume growth exponent of closed subgroups of the noncompact semisimple Lie group $G$. As it turns out, for a discrete subgroup, it can be recovered from the growth indicator function, while for a reductive subgroup, we relate this quantity to the Lie algebras, or more precisely to the local volume decay exponent $\betagh$.

Recall that the symmetric measure on a locally compact group $H$ is given by $\dif \nu_H(h) = \Delta_H(h)^{-\frac{1}{2}} \dif h$. If the group $H$ is either reductive or discrete, then $\nu_H$ coincides with the Haar measure.
\begin{definition}\label{def: relative exponent of closed subgroups}
    Define the \emph{relative volume growth exponent} of a closed subgroup $H$ in $G$ by
    \begin{equation*}
        \delta_{G/H} := \maxde{0, \,\sup_{B\Subset G} \limsup_{g\to\infty} \frac{\log \nu_H\!\left(H\cap BgB\right)}{\log\nu_G(BgB)}},
    \end{equation*}
    where $B$ ranges over compacta of $G$ of nonempty interior.
\end{definition}

\begin{remark}\label{rem: cofinal choice of B in relative exponent}
    (1) A finite-covering argument shows that the limsup in the definition does not depend on the choice of $B$.

    (2) One can also restrict $g\to\infty$ to be inside $H$.
\end{remark}

Consider now the following analogue of Poincaré series
\begin{equation}
    I(t):= \int_H e^{-2t\rho\kappa(h)} \dif\nu_H(h).
\end{equation}
The following proposition characterizes the growth exponent $\deltagh$ as the abscissa of convergence of $I$.

\begin{proposition}\label{prop: deltagh as the abscissa of convergence}
    For any closed subgroup $H$ of a connected noncompact semisimple Lie group $G$ of finite center, we have
    \begin{equation*}
        \deltagh = \inf\setdef{t\in [0,\infty]: I(t)<\infty}.
    \end{equation*}
\end{proposition}
\begin{proof}
    Denote the right hand side by $\delta_0$.
    First show $\deltagh\leq \delta_0$. Let $B\Subset G$. Then by \autoref{lem: compact error in cartan projection}, there exists $R>0$ such that $\kappa(BgB)\subset \kappa(g) + \alie(R)$ for all $g\in G$, whence for all $h\in H\cap BgB$ we have
    \begin{equation*}
        2\rho\kappa(h) \leq 2\rho\kappa(g) + C_1,
    \end{equation*}
    where the constant $C_1= \sup_{\alie(R)} 2\rho$ is uniform. Since $I(t)<\infty$ for any $t>\delta_0$, we then have
    \begin{align*}
        \nuh{H\cap BgB} &\leq C_2 \,e^{2t\rho\kappa(g)} \int_{H\cap BgB} e^{-2t\rho\kappa(h)} \dif\nu_H(h)\\
        &\leq C_2 \,e^{2t\rho\kappa(g)} I(t) \leq C_3 \, e^{2t\rho\kappa(g)},
    \end{align*}
    uniformly for all $g\in G$. From the definition of the growth exponent $\deltagh$ we deduce that $t\geq \deltagh$, whence we have $\delta_0\geq \deltagh$.

    Next we show $\deltagh\geq \delta_0$. 
    Let $\calL$ be a lattice of $\alie$ and $\calL^+=\calL\cap \alie^+$. Then there exists some $r>0$ such that $\alie^+$ is covered by the balls of radius $r$ centered at elements in $\calL^+$. Fix any number $t> \deltagh$ and then fix a small number $\varepsilon>0$ so that $t-\varepsilon > \deltagh$.

    Define the subset $B= K e^{\alie(r)} K \Subset G$. \autoref{lem: compact error in cartan projection} yields a constant $c>0$ such that whenever $h\in Be^X B$, we have $2\rho(X) \leq 2\rho\kappa(h) + c$, whence
    \begin{align*}
        \int_{H\cap Be^X B} e^{-2t\rho\kappa(h)} \dif\nu_H(h) \leq C_4\, e^{-2t\rho(X)} \nuh{H\cap Be^X B}
    \end{align*}
    uniformly for all $X\in \alie^+$. But since $t-\varepsilon>\deltagh$, we deduce from \autoref{prop: BgB volume growth in G} that
    \begin{equation*}
        \nuh{H\cap Be^X B} \leq C_5\, e^{2(t-\varepsilon)\rho(X)}
    \end{equation*}
    uniformly for $X\in\alie^+$, whence for all $X\in\alie^+$ we have
    \begin{equation*}
        \int_{H\cap Be^X B} e^{-2t\rho\kappa(h)} \dif\nu_H(h) \leq C_6\, e^{-2\varepsilon\rho(X)}.
    \end{equation*}
    Since the construction implies that the subset $B e^X B$ contains $K e^{X+\alie(r)} K$ for all $X\in\alie^+$, we obtain
    \begin{align*}
        I(t) \leq \sum_{X\in\calL^+} \int_{H\cap Be^X B} e^{-2t\rho\kappa(h)} \dif\nu_H(h) \leq C_6 \sum_{X\in\calL^+} e^{-2\varepsilon\rho(X)},
    \end{align*}
    which is finite as the lattice $\calL$ grows polynomially. Hence, we have $t>\delta_0$ and we obtain $\deltagh\geq \delta_0$.
\end{proof}

\subsubsection{Reductive subgroups}
\begin{proposition}\label{prop: delta=rho for reductive subgroups}
    If $H$ is a reductive subgroup of $G$, then
    \begin{equation*}
        \deltagh = \rhogh.
    \end{equation*}
\end{proposition}

First we introduce the notation.
Let $A_H$ be a Cartan subspace of $H$. Extending $A_H$ to a Cartan subspace $A$ of $G$, we have $A_H = A\cap H$. By the reductiveness of $H$, there exists a Cartan decomposition $G=KAK$ of $G$ such that the subgroup $K_H:= K\cap H$ is a maximal compact subgroup of $H$, with the Cartan decomposition of $H$ given by $H=K_H A_H K_H$.

Let $\alie, \alieh$ denote respectively the Cartan subspaces and $W_G, W_H$ denote respectively the associated Weyl groups. Since the Killing form is adjoint-invariant, the induced Euclidean norm $\norm{\cdot}$ on $\alie$ is $W_G$-invariant and its restriction to $\alieh$ is $W_H$-invariant.

By the Cartan decomposition of $H$, we have for all $h\in H$ that
\begin{equation}\label{eq ppty: kappaH and kappa mod wg}
    h \in K_H \bigexp{\setdef{w\kappa(h)\colon w\in W_G}\cap \alieh} K_H.
\end{equation}

\begin{lemma}\label{lem: upper bound on volume growth in H}
    For any $B\Subset G$, there exists a constant $C>0$ such that uniformly for all $g\in G$,
    \begin{equation*}
        \nuh{H\cap B g B } \leq Ce^{2\rhogh\rho\kappa(g)}.
    \end{equation*}
\end{lemma}
\begin{proof}
    Write $X=\kappa(g)$. By \autoref{lem: compact error in cartan projection}, there exists a constant $r>0$ such that
    $ \kappa(BgB) \subset \kappa(g) + \alie(r)$ for all $g\in G$. 
    Then \eqref{eq ppty: kappaH and kappa mod wg} yields
    \begin{equation}\label{eq intm: kappa H of BgB}
        H\cap BgB \subset \bigcup_{w\in W_G} K_H e^{(wX+\alie(r)) \cap \alieh} K_H.
    \end{equation}
    Hence, by \autoref{prop: KAK Cartan integral} and \autoref{eg: rhoh in case of adjoint rep}, we have
    \begin{equation}\label{eq intm: upper bound of volume part in H}
        \nuh{H\cap BgB} \leq \sum_{w\in W_G} \int_{(wX+\alie(r))\cap \alieh} \bigexp{\rhoh(Y)} \dif Y.
    \end{equation}
    Since $\rhoh\leq \rhogh \rhog$ by the definition of $\rhogh$, and since $\rhog$ is $W_G$-invariant and uniformly Lipschitz on $\alieh$,
    we have 
    \begin{align*}
        \bigexp{\rhoh(Y)}&\leq \bigexp{\rhogh\rhog(Y)} \leq C_1 \bigexp{\rhogh\rhog(wX)}\\
        &= C_1\bigexp{\rhogh\rhog(X)} = C_1\bigexp{2\rhogh\rho(X)}
    \end{align*}
    uniformly for $Y\in wX+\alie(r)$. We conclude by feeding back to \eqref{eq intm: upper bound of volume part in H}.
\end{proof}

\begin{proof}[Proof of \autoref{prop: delta=rho for reductive subgroups}]
    For $B\Subset G$,
    \autoref{prop: BgB volume growth in G} and \autoref{lem: upper bound on volume growth in H} yield
    \begin{equation*}
        \frac{\log \nuh{H\cap BgB}}{\log \nug{BgB}} \leq \frac{C_1 + 2\rhogh \rho\kappa(g)}{C_2 + 2\rho\kappa(g)}
    \end{equation*}
    uniformly for all $g\in G$. Since $\rho\kappa(g)\to \infty$ as $g\to \infty$, we get $\deltagh\leq \rhogh$.
    
    To show $\deltagh\geq \rhogh$, let $B\Subset G$ and $B_H := B\cap H \Subset H$. Then the intersection $H\cap BhB$ contains $B_H h B_H$ for all $h\in H$. Now \autoref{prop: BgB volume growth in G} applies to the real reductive group $H$ without modification, whence 
    $$\nuh{H\cap Be^X B} \geq ce^{\rhoh(X)}$$
    for all $X\in\alieh$. 
    By the continuity of the homogeneous functions $\rhoh,\rhog$ on $\alieh$, there exists a nonzero vector $X\in\alieh$ satisfying
    $ \rhogh\,\rhog(X) = \rhoh(X)$.
    Setting $g_n := e^{nX} \in H$ with $g_n\to \infty$ in $G$, we obtain
    \begin{equation*}
        \deltagh \geq \limsup_{n\to\infty} \frac{\log \nuh{H\cap Be^{nX} B}}{\log \nug{Be^{nX} B}} \geq \limsup_{n\to\infty} \frac{\log \left(c e^{n\rhoh(X)}\right)}{\log \left(C e^{n\rhog(X)}\right)} =\rhogh
    \end{equation*}
    by applying the lower bound of \autoref{prop: BgB volume growth in G} to $H$ and the upper bound to $G$.
\end{proof}

\subsubsection{Discrete subgroups}
For a discrete subgroup $\Gamma$,
recall that Quint's growth indicator function $\psi_\Gamma: \alie^+\to \bbR\cup\setdef{-\infty}$ is defined by
\begin{equation*}
    \psi_\Gamma(X) := \norm{X} \inf_{\calC\ni X} \inf \setdef{t\in \bbR : \sum_{\gamma\in\Gamma, \,\kappa(\gamma)\in\calC} e^{-t\norm{\kappa(\gamma)}}<\infty},
\end{equation*}
where $\calC$ ranges over open cones in $\alie$ which contain $X$.
One has (cf.\,\cite[\S I.1]{quint2002divergence})
\begin{equation*}
    \sup_{\alie^+} \frac{\psi_\Gamma}{2\rho} = \inf\setdef{t\in\bbR: \sum_{\gamma\in\Gamma} e^{-2t\rho\kappa(\gamma)} < \infty}.
\end{equation*}
We have the following immediate consequence of \autoref{prop: deltagh as the abscissa of convergence}.

\begin{proposition}\label{prop: delta = max of psi/rho, discrete subgroups}
    If $\Gamma$ is a discrete subgroup of $G$, then
    \begin{equation*}
        \deltaggamma = \maxde{\sup_{X\in \alie^+} \frac{\psi_\Gamma(X)}{2\rho(X)}, \,0}.
        \qedineq
    \end{equation*}
\end{proposition}

\begin{example}
    When $\Gamma$ is a discrete subgroup of $G=\SL(2,\bbR)$, the exponent $\delta_{G/\Gamma}$ coincides with the usual critical exponent $\delta_\Gamma$. In general, when the semisimple group $G$ is of real rank one, these two exponents are related by the equality $\delta_\Gamma = d_G \deltaggamma$ where the constant $d_G$ depends only on $G$.
\end{example}

\subsubsection{Volume exponents below $1/2$}
In \autoref{sec unif decay} infra we will prove \autoref{mythm: rho theta delta} which implies that for a closed connected subgroup $H$, if $\deltagh>1/2$, then the two volume exponents $\deltagh = \betagh$ are equal. The following example provides a counterexample to this equality below $1/2$.

\begin{example}\label{eg: delta ne beta for G/N}
    Let $G=\SL(2,\bbR)$ and $N=\setdef{n_s=\begin{pmatrix} 1 & s\\ & 1\end{pmatrix}: s\in \bbR}$ be a unipotent subgroup. Then clearly $\beta_{G/N}=0$. 
    
    Yet we claim $\delta_{G/N}=1/2$.
    Indeed, computing the singular values gives
    $$e^{2\rho\kappa(n_s)} = \frac{s^2 + 2 + \sqrt{s^4 + 4s^2}}{2}, \;\; \forall s\in \bbR.$$
    Thus we can compute the relative growth exponent $\delta_{G/N}=1/2$ through the ``Poincaré integral'' (\autoref{prop: deltagh as the abscissa of convergence}):
    \begin{equation*}
        I(t) = \int_{\bbR} \left( \frac{2}{s^2 + 2 + \sqrt{s^4 + 4s^2}}\right)^t \dif s <\infty \iff t > 1/2.
    \end{equation*}
\end{example}

% section
\section{Decay of coefficients and volume growth}\label{sec decay and vol growth}
The goal of this section is to prove \autoref{mythm: theta and delta} on the following relationship
\begin{equation*}
    \thetagh = \deltagh = 1 - \frac{1}{p_{G/H}} \geq \betagh
\end{equation*}
between the four exponents we have defined.

Since uniform decay is a strong property, the upper bounds on the other exponents by the coefficient decay exponent $\thetagh$ are less complicated. The main difficulty, which we will start with, is to establish uniform decay estimates from other data, for which our method is inspired by \cite{lutsko-weich2024polyhedral}.

Let $G$ be a real semisimple algebraic group and $H$ be a closed subgroup of $G$.
Then by \autoref{lem: delta function and quasi-invariant measure on homogeneous space}, the homogeneous space $G/H$ admits a $G$-quasi-invariant Radon measure $\dif\mu(gH)$ and a continuous density function $\delta: G\to \bbR^+$ which satisfy \eqref{eq ppty: delta-function, measures on homogeneous spaces} and \eqref{eq ppty: integration on homogeneous space} in \autoref{lem: delta function and quasi-invariant measure on homogeneous space}.

\subsection{Matrix coefficients of induced representations}\label{subsec: transform}
As a preliminary step, we transform the matrix coefficients of $\Ltwogh$ into more accessible terms.
For later applications as well, we address more generally the coefficients of an induced unitary representation $(\pi,\Hil)=\Ind_H^G(\sigma, \calV)$.
Setting $\sigma=1_H$ will recover $\Ltwogh$.

Given functions $\varphi_1, \varphi_2 \in\Cc(G)$ and vectors $v_1,v_2\in\calV$, consider the elements $f_i := \rmI(\varphi_i, v_i)\in\Hil$ for $i=1,2$ (constructed in \eqref{eq def: I(phi, v)}). 
To study the matrix coefficient $\Innerprod{\pi(\cdot)f_1,f_2}$, we first expand it into integrals on groups by using the expressions of $f_1,f_2$.
We deduce
\begin{align*}
    &\Innerprod{\pi(g)f_1, f_2}\\
    &= \int_{G/H} \Innerv{f_1(g^{-1}x), f_2(x)} \left(\frac{\delta(g^{-1}x)}{\delta(x)}\right)^{\frac{1}{2}} \dif\mu(xH) &(\textup{def of }\Hil)\\
    &= \int_{G/H} \int_H \Innerv{\varphi_1(g^{-1}xh)\sigma(h)v_1, f_2(x)} \left(\frac{\delta(g^{-1}x)}{\delta(x)}\right)^{\frac{1}{2}} \dif h \dif\mu(xH) &(\textup{def of } f_1)\\
    &= \int_{G/H} \int_H \Innerv{v_1, f_2(xh)}\varphi_1(g^{-1}xh)\left(\frac{\delta(g^{-1}xh)}{\delta(xh)}\right)^{\frac{1}{2}} \dif h \dif \mu(xH) & (\sigma\textup{-eqv})\\
    &= \int_{G} \Innerv{v_1, f_2(x)}\varphi_1(g^{-1}x)\left({\delta(g^{-1}x)}{\delta(x)}\right)^{\frac{1}{2}}\dif x & (\textup{by \eqref{eq ppty: integration on homogeneous space}})\\
    &= \int_{G} \int_H \Innerv{v_1, \sigma(h)v_2} \varphi_1(g^{-1}x)\overline{\varphi_2(xh)} \left({\delta(g^{-1}x)}{\delta(x)}\right)^{\frac{1}{2}} \dif h \dif x & (\textup{def of }f_2)\\
    &= \int_H \Innerv{v_1,\sigma(h)v_2}\int_G \varphi_1(g^{-1}x) \overline{\varphi_2(xh)} \left({\delta(g^{-1}x)}{\delta(x)}\right)^{\frac{1}{2}}\dif x \dif h & (\textup{Fubini})\\
    &= \int_H \Innerv{v_1,\sigma(h)v_2}\int_G \varphi_1(g^{-1}x) \overline{\varphi_2(xh)} \left({\delta(g^{-1}x)}{\delta(xh)}\right)^{\frac{1}{2}}\dif x \dif\nu_H(h), & (\textup{prop of }\delta)
\end{align*}
where we recall that $\dif\nuh{h} = \Delta_H(h)^{-1/2}\dif h$ is the symmetric measure on $H$.
By changing $h$ to $h^{-1}$, we obtain
\begin{equation}\label{eq res: transform of pi coefficients}
    \Innerprod{\pi(g)f_1, f_2} = \int_H \Innerv{\sigma(h)v_1, v_2} \Phi(h,g) \dif\nu_H(h),
\end{equation}
where we define
\begin{equation}\label{eq def: Phi(h,g)}
    \Phi(h, g) := \int_G \varphi_1(g^{-1}x) \overline{\varphi_2(xh^{-1})} \left({\delta(g^{-1}x)}{\delta(xh^{-1})}\right)^{\frac{1}{2}}\dif x.
\end{equation}
Using the volume decay in $G$ studied in the last section, we can now obtain the first estimates for decay of coefficients.
\begin{lemma}\label{lem: first majoration of decay of coefficients}
    Let $B_G$ be the open neighborhood of $e\in G$ given in \autoref{lem: open nbhd B_G, shrink NMAN}.
    For any $v_1,v_2\in \calV$ and any $\varphi_1,\varphi_2\in \Cc(G)$ whose supports $B_i:=\supp\varphi_i$ satisfy $B_iB_i^{-1}\subset B_G$, let $f_i:=\rmI(\varphi_i,v_i)\in \Hil$ for $i=1,2$. Then there exists a constant $C>0$ such that we have uniformly for all $g\in G$,
    \begin{equation*}
        \Modofinner{\pi(g)f_1, f_2} \leq C e^{-2\rho\kappa(g)} \int_{H\cap (B_2^{-1}g B_1)} \Modinnerv{\sigma(h)v_1, v_2} \dif\nu_H(h).
    \end{equation*}
\end{lemma}
\begin{proof}
    If $\Phi(h,g)\ne 0$, then there exists $x_0\in G$ such that $g^{-1}x_0=:b_1\in B_1$ and $x_0h^{-1}=:b_2\in B_2$,
    so $h=b_2^{-1}g b_1 \in B_2^{-1}g B_1$.
    By using the unimodularity of $G$, we deduce
    \begin{align}
        \Phi(h,g) &= \int_G \varphi_1(g^{-1}x) \overline{\varphi_2(xb_1^{-1}g^{-1}b_2)}\left({\delta(g^{-1}x)}{\delta(xh^{-1})}\right)^{\frac{1}{2}} \dif x\nonumber\\
        &= \int_G \varphi_1(g^{-1}ygb_1) \overline{\varphi_2(yb_2)} \left({\delta(g^{-1}ygb_1)}{\delta(yb_2)}\right)^{\frac{1}{2}}\dif y\label{eq intm: transform of Phi(h,g)}
    \end{align}
    by setting $y=xb_1^{-1}g^{-1} = x h^{-1} b_2^{-1}$.
    For $i=1,2$, define
    \begin{equation*}
        \tilde{\varphi}_i(x) := \sup_{b\in B_i} \norm{\varphi_i(xb)}\delta(xb)^{\frac{1}{2}}.
    \end{equation*}
    Then $\tilde{\varphi}_i \in \Cc(B_G)$ by the hypothesis on $B_i$. Applied to $\tilde{\varphi}_1, \tilde{\varphi}_2$, \autoref{prop: local volume decay in G} yields
    \begin{equation*}
        \int_G \tilde{\varphi}_1(g^{-1}yg) \tilde{\varphi}_2(y) \dif y \leq Ce^{-2\rho\kappa(g)}
    \end{equation*}
    uniformly for all $g\in G$. Feeding back to \eqref{eq intm: transform of Phi(h,g)}, we deduce
    \begin{equation*}
        \norm{\Phi(h,g)} \leq Ce^{-2\rho\kappa(g)} \indicator_{B_2^{-1}gB_1}(h)
    \end{equation*}
    uniformly for $g\in G$ and $h\in H$. We conclude by feeding back to \eqref{eq res: transform of pi coefficients}.
\end{proof}

\subsection{From volume growth to uniform decay}
\begin{proposition}\label{prop: theta <= delta}
    $\thetagh\leq \deltagh$.
\end{proposition}
\begin{proof}
    By \autoref{cor: I(varphi,1) spans Cc(GH)}, the subset
    \begin{equation*}
        \setdef{\rmI(\varphi,1): \varphi\in\Cc(G) \textup{ with } (\supp\varphi)(\supp\varphi)^{-1}\subset B_G}
    \end{equation*}
    spans the space $\Cc(G/H)$. By applying \autoref{lem: first majoration of decay of coefficients} to the induced representation $\lambdagh=\Ind_H^G 1_H$, we deduce for any $f_1,f_2\in \Cc(G/H)$ that
    \begin{equation*}
        \Modofinner{\lambdagh(g)f_1, f_2} \leq C_1 e^{-2\rho\kappa(g)} \nuh{H\cap B_2^{-1}g B_1}
    \end{equation*}
    uniformly for $g\in G$, for some compacta $B_1,B_2$ of $G$.
    Choose $B\Subset G$ which contains $B_2^{-1} \cup B_1$. By the definition of $\deltagh$ and by \autoref{prop: BgB volume growth in G}, given any number $\delta>\deltagh$, there exists a constant $C_1>0$ such that
    \begin{equation*}
        \nuh{H\cap B_2^{-1}g B_1}\leq \nuh{H\cap BgB} \leq C_2 e^{2\delta\rho\kappa(g)}
    \end{equation*}
    for all $g\in G$, whence
    \begin{equation*}
        \Modofinner{\lambdagh(g)f_1, f_2} \leq C_3 e^{-2(1-\delta)\rho\kappa(g)}
    \end{equation*}
    uniformly for $g\in G$. 
    Since this is valid for any functions $f_1,f_2\in \Cc(G/H)$ and any number $\delta>\deltagh$, we have $\deltagh \geq \thetagh$.
\end{proof}

\subsection{From uniform decay to integrability}
Recall that the integrability exponent $p_{G/H}$ optimizes the condition that for all $f_1,f_2\in\Cc(G/H)$ (or equivalently in $\Lspace^\infty_{\mathrm{c}}(G/H)$), we have
\begin{equation*}
    \Innerprod{\lambdagh(\cdot)f_1, f_2} \in \bigcap_{p>p_{G/H}}\Lspace^p(G).
\end{equation*}

\begin{proposition}\label{prop: p <= theta}
    $1- 1/p_{G/H}\leq \thetagh$.
\end{proposition}
\begin{proof}
    Suppose $\thetagh<1$.
    For any $\theta> \thetagh$ and any $f_1,f_2\in\Cc(G/H)$, we obtain from \autoref{def: uniform decay exponent} that
    \begin{equation*}
        \Modofinner{\lambdagh(g)f_1,f_2} \leq Ce^{(2\theta-2)\rho\kappa(g)}
    \end{equation*}
    uniformly for $g\in G$. Then by applying \autoref{prop: KAK Cartan integral}, we have
    \begin{equation*}
        \int_G \Modofinner{\lambdagh(g)f_1,f_2}^{p} \dif g \leq C \int_{\alie^+} e^{(2\theta p-2p + 2)\rho(X)} \dif X,
    \end{equation*}
    which is finite as long as $\theta<1-1/p$. By the arbitrariness of $\theta>\thetagh$, we conclude that $1-1/p_{G/H} \leq \thetagh$.
\end{proof}

\subsection{From integrability to volume growth}
\begin{lemma}\label{lem: transform of coefficient for 1B}
    Let $B$ be a symmetric compact neighborhood of $e$ in $G$. For 
    \begin{equation*}
        f(xH) = \int_H \indicator_{B}(xh)\delta(xh)^{-\frac{1}{2}}\dif h \in \Lspace_{\mathrm{c}}^\infty(G/H),
    \end{equation*}
    we have uniformly for all $g\in G$ that
    \begin{equation*}\label{eq res: transform}
        \int_{BBgBB} \Innerprod{\lambdagh(x)f, f} \dif x \geq \nug{B}^2 \nuh{H\cap BgB},
    \end{equation*}
    where $\dif x$ denotes the Haar measure of $G$.
\end{lemma}
\begin{proof}
    Let us apply the preliminary computations in \autoref{subsec: transform} to the induced representation $\lambdagh = \Ind_H^G 1_H$ and the functions
    $$\varphi_1(x) = \varphi_2(x) = \indicator_{B}(x) \delta(x)^{-\frac{1}{2}}\in \Lspace^\infty_{\mathrm{c}}(G),$$
    so that the functions $f_1,f_2$ therein coincide with $f$.
    Now \eqref{eq def: Phi(h,g)} becomes
    $\Phi(h,g) = \nug{gB \cap Bh}$
    and thus \eqref{eq res: transform of pi coefficients} yields
    \begin{align*}
        \Innerprod{\lambdagh(g)f_1, f_2} = \int_{H} \nug{gB\cap Bh} \dif\nu_H(h).
    \end{align*}
    Let $B_g:= BBgBB$ for $g\in G$. By Fubini's theorem, we have
    \begin{align}
        &\int_{B_g} \Innerprod{\lambdagh(x)f_1, f_2} \dif x = \int_{B_g} \int_{H} \nug{xB\cap Bh}  \dif\nu_H(h) \dif x\nonumber\\
        &\geq \int_{H\cap BgB} \int_G \indicator_{B_g}(x) \nug{xB\cap Bh}\dif x \dif \nuh{h}.\label{eq intm: uniform lower bound of coefficient, but with int}
    \end{align}
    Given any $h\in H\cap BgB$, we can write $h= b_1gb_2$ for some $b_1,b_2\in B$.
    Then from Fubini's theorem, the symmetric property of the subset $B$, and the unimodularity of the group $G$, we deduce
    \begin{align*}
        \int_G \indicator_{B_g}(x) \nug{xB\cap Bh}\dif x &= \int_G \int_G \indicator_{B_g}(x) \indicator_{xB}(y) \indicator_{Bb_1gb_2}(y) \dif y \dif x\\
        &= \int_G \nug{BBgBB \cap yB} \indicator_{Bb_1gb_2}(y) \dif y\\
        &\geq \int_{Bb_1 gb_2} \nug{B}\dif y = \nug{B}^2,
    \end{align*}
    whence we conclude by feeding back to \eqref{eq intm: uniform lower bound of coefficient, but with int}.
\end{proof}

\begin{proposition}\label{prop: p >= delta}
    $\deltagh\leq 1-1/p_{G/H}$.
\end{proposition}
\begin{proof}
    By the definition of the relative volume growth exponent $\deltagh$ and \autoref{rem: cofinal choice of B in relative exponent}, for any number $\delta < \deltagh$ there exists a compact, symmetric, $K$-bi-invariant neighborhood $B$ of $e$ in $G$, and a sequence $(X_n)$ in $\alie^+$ going to infinity, such that
    \begin{equation*}
        \log \nuh{H\cap Be^{X_n}B} \geq \delta \log \nug{Be^{X_n} B}
    \end{equation*}
    uniformly for $n\in\bbN$, whence we can deduce from \autoref{prop: BgB volume growth in G} that
    \begin{equation}\label{eq intm: lower bound on volume growth xn}
        \nuh{H\cap Be^{X_n}B} \geq c_1 e^{2\delta\rho(X_n)}
    \end{equation}
    uniformly for $n\in\bbN$. By passing to a subsequence, we can assume that the subsets $BBe^{X_n} BB$ with $n\in\bbN$ are pairwise disjoint.

    Let the function $f\in\Lspace^2(G/H)$ be given as in \autoref{lem: transform of coefficient for 1B}.
    We proceed to study the integrability of the matrix coefficient $\Innerprod{\lambdagh(\cdot)f,f}$. By applying the H\"older inequality, we obtain for $p>1$ that
    \begin{equation*}
        \int_{BBgBB} \Modofinner{\lambdagh(x)f,f}^p\dif x \geq \nug{BBgBB}^{1-p} \left(\int_{BBgBB}\Innerprod{\lambdagh(x)f, f}\dif x\right)^p.
    \end{equation*}
    Further from \autoref{lem: transform of coefficient for 1B} and \autoref{prop: BgB volume growth in G}, we deduce
    \begin{equation*}
        \int_{BBgBB} \Modofinner{\lambdagh(x)f,f}^p\dif x\geq c_2 e^{2(1-p)\rho\kappa(g)} \nuh{H\cap BgB}^p
    \end{equation*}
    uniformly for all $g\in G$.
    By setting $g=e^{X_n}$ and applying \eqref{eq intm: lower bound on volume growth xn}, we obtain
    \begin{equation*}
        \int_{BBe^{X_n}BB} \Modofinner{\lambdagh(x)f,f}^p\dif x \geq c_3 e^{(2-2p+2p\delta)\rho(X_n)}
    \end{equation*}
    uniformly for $n\in\bbN$. If $\Innerprod{\lambdagh(\cdot)f,f}\in\Lspace^p(G)$, then we have
    \begin{equation*}
        \sum_{n\in\bbN} e^{2(1-p+p\delta)\rho(X_n)} < \infty,
    \end{equation*}
    which implies that $1-1/p>\delta$. By the arbitrariness of $p>p_{G/H}$ and $\delta<\deltagh$, we conclude that $1-1/p_{G/H}\geq \deltagh$.
\end{proof}

\subsection{Local lower bound of coefficient decay}\label{subsec: beta leq theta}
\begin{proposition}\label{prop: beta <= theta}
    For any closed subgroup $H$ of $G$, we have
    $$ \betagh \leq \thetagh.$$
\end{proposition}

\begin{proof}
The philosophy of the proof is that the local dynamics near $eH\in G/H$ mimics the dynamics on $\glie/\hlie$.
We shall incarnate this idea, even though generally we have no canonical local chart for the homogeneous space $G/H$.
Let us decompose the proof into several steps.

\vspace*{4pt}
\noindent \emph{Step 1. Preliminary reduction.}

It suffices to show that $\rho_\hlie(X) \leq \thetagh \rho_\glie(X)$ for every element $X\in\hlie$.
From now on we fix a nonzero $X\in \hlie$ which, by the Jordan decomposition, splits uniquely as a commuting sum $X=X_s+X_n$ of a semisimple element $X_s$ and a nilpotent element $X_n$ inside $\glie$, meaning that the operators $\ad X_s, \ad X_n \in \GL(\glie)$ are respectively semisimple, nilpotent. In general, the element $X_s$ may not be in $\hlie$, but we always have $[X_s, \hlie] \subset \hlie$ since $\ad X_s$ can be written as a real polynomial of $\ad X$.
We have $\rho_\hlie(\ad X_s) = \rho_\hlie(X)$ and $\rho_\glie(X_s) = \rho_\glie(X)$.
Write $x_s := \ad X_s$ and $x_n := \ad X_n$.

We henceforth fix a norm $\norm{\cdot}$ on $\glie$.
By the nilpotency of $x_n\in \End(\glie)$, for some constant $C\geq 1$ and $d=\dim\glie$ we have
\begin{equation}\label{eq lem: xn grows poly t}
    \opnorm{\exp(-tx_n)} \leq Ct^d, \;\;\textup{for all } t\geq 1.
\end{equation}

% \vspace*{4pt}
\noindent \emph{Step 2. A choice of the local chart.}

By the semisimplicity of $x_s$, there is an $x_s$-invariant complementary subspace $\wlie$ with $\glie = \hlie \oplus \wlie$, but $\wlie$ is not necessarily invariant by $\ad X$ or $x_n$.
Now the space $\wlie$ is naturally identified with $\glie/\hlie$, with the same action of $x_s$, and we will use the local chart $E: \wlie \to G/H$ with $E(Y) := e^Y H$.
We fix a convex compact neighborhood $D$ of $0$ in $\wlie$ such that $E$ restricts a homeomorphism onto its image denoted by $B\Subset G/H$. Note that
\begin{equation}\label{eq lem 1: local exp}
    \textup{there exists $r>0$, such that $e^Y H\in B$ for all $Y\in \glie$ with $\norm{Y}\leq r$.}
\end{equation}
Up to shrinking $D$, we can assume that the elements in $D$ have norm less than $r$ and we have $E(D)\subset B$ instead of equality.

\vspace*{4pt}
\noindent \emph{Step 3. A technical lemma.}
\begin{lemma}\label{lem: lower bound on G/H, polynomial error}
    For all $t\geq 1$, let $D_t:= \frac{1}{Ct^d} D \Subset \wlie$. Then we have
    $$  E\!\left( \exp(t x_s) D_t \cap D_t\right) \subset  e^{tX}B \cap B. $$
\end{lemma}
\begin{proof}[Proof (lemma)]
    Let $Y\in \exp(t x_s) D_t \cap D_t$. Then we have $e^Y H \in B$ since $D_t \subset D$. It suffices to prove that $e^{-tX}e^Y H \in B$. By assumption, there exists $Y_0\in D_t$ with $Y = e^{t x_s} Y_0$. We deduce $\Ad(e^{-tX})Y = e^{-tx_n}Y_0$ and furthermore by \eqref{eq lem: xn grows poly t} that 
    $$\norm{\Ad(e^{-tX})Y} \leq \opnorm{\exp(-tx_n)}\norm{Y_0} \leq r.$$
    By \eqref{eq lem 1: local exp} and the fact that $X\in\hlie$, we have
    \begin{equation*}
        e^{\Ad(e^{-tX})Y}H = e^{-tX}e^Y H \in B. \qedhere
    \end{equation*}
\end{proof}

\noindent \emph{Step 4. Local decay from linear algebra.}

By \autoref{lem: lower bound on G/H, polynomial error} and by taking into account the Radon--Nikodym derivatives, we deduce, uniformly for $t\geq 1$ that
\begin{equation*}
    \Innerprod{\lambda_{G/H}(e^{tX})\indicator_{B}, \indicator_{B}}  \geq c_1\, e^{-\Tr_{\wlie}(tx_s)/2}\Vol_{\wlie}\!\left(e^{tx_s}D_t\cap D_t\right),
\end{equation*}
with $D_t = \frac{1}{Ct^d} D$. By \autoref{cor: rho fun semisimple and volume decay}, for the constant $d_1:= d \dim\wlie$ and for all $t\geq 1$ we have uniformly,
\begin{equation*}
    e^{-\Tr_{\wlie}(tx_s)/2}\Vol_{\wlie}\!\left(e^{tx_s}D_t\cap D_t\right) \geq c_2 t^{-d_1} e^{-t\rho_{\wlie}(x_s)}.
\end{equation*}
But since $\rho_\wlie(x_s)=\rho_\wlie(X)$ (by $[x_s,x_n]=0$), we have for all $t\geq 1$,
\begin{equation*}
    \Innerprod{\lambda_{G/H}(e^{tX})\indicator_{B}, \indicator_{B}}  \geq c_3 t^{-d_1} e^{-t\rho_{\wlie}(X)}.
\end{equation*}
But by the definition of $\thetagh$, for any $\theta>\thetagh$ we have uniformly for $t>0$,
\begin{equation*}
    \Innerprod{\lambda_{G/H}(e^{tX})\indicator_{B}, \indicator_{B}} \leq C_4 e^{(\theta-1)\rho_\glie(tX)}.
\end{equation*}
By sending $t\to +\infty$ in $t^{-d_1} e^{-t\rho_{\wlie}(X)}\leq C_5 e^{(\theta-1)t\rho_\glie(X)}$, we get
$$ -\rho_\glie(X) + \rho_{\hlie}(X) = \rho_{\wlie}(X) \leq (\theta-1)\rho_\glie(X),$$
whence $\rho_\hlie(X) \leq \theta_{G/H}\rho_\glie(X)$, completing the proof of \autoref{prop: beta <= theta}.
\end{proof}

\subsection{Conclusion of proofs}
\begin{proof}[Proof of \autoref{mythm: theta and delta}]
    The combination of \autoref{prop: theta <= delta}, \autoref{prop: p <= theta}, \autoref{prop: p >= delta}, and \autoref{prop: beta <= theta} yields
    \begin{align*}
        \deltagh \geq \thetagh \geq 1- \frac{1}{p_{G/H}} \geq \deltagh,\;\; \thetagh \geq \betagh.
    \end{align*}
    Hence, we conclude the proof of the theorem.
\end{proof}

Applications to special classes of subgroups are immediate.

\begin{proof}[Proof of \autoref{mythm: reductive subgp, delta=theta=rho}]
    When the subgroup $H$ is reductive, it follows from \autoref{mythm: theta and delta} and \autoref{prop: delta=rho for reductive subgroups} that
    \begin{equation*}
        \thetagh = \deltagh = \rhogh.
        \qedhere
    \end{equation*}
\end{proof}

\begin{proof}[Proof of \autoref{mythm: discrete subgp, delta=theta=}]
    When the subgroup $\Gamma$ is discrete, it follows from \autoref{mythm: theta and delta} and \autoref{prop: delta = max of psi/rho, discrete subgroups} that
    \begin{equation*}
        \theta_{G/\Gamma} = \deltaggamma = \maxde{\sup_{\alie^+} \frac{\psi_\Gamma}{2\rho},\, 0}.
        \qedhere
    \end{equation*}
\end{proof}

% section
\section{Uniform decay of induced representations}\label{sec unif decay}
The goal of this section is to prove \autoref{mythm: rho theta delta}. Thanks to \autoref{mythm: theta and delta}, it will suffice to show that above $1/2$, the coefficient decay exponent $\thetagh$ is bounded from above by the local volume decay exponent $\betagh$.

We will first prove the statement for algebraic subgroups (\textcolor{blue}{Sections}~\ref{subsec: ingredients bk}--\ref{subsec: proof for alg subgp}) and then extend to all closed subgroups with finitely many components (\autoref{subsec: extend to connected subgroups}).
We will establish uniform decay estimates by applying the strategy of Benoist--Kobayashi in \cite{benoist-kobayashi2022tempered2}.
More specifically, through a chain of algebraic subgroups, we will reduce the unitary induction from $H$ to $G$ into the composition of several ones which are easier to handle. 
The rich structure theory of semisimple groups allows us to absorb the local data into each intermediate induced representation.

\subsection{Ingredients from Benoist--Kobayashi}\label{subsec: ingredients bk}
From now on, we fix $H$ to be an algebraic subgroup of the real semisimple algebraic group $G$.
Since the values of the four exponents remain unchanged if we pass to an open subgroup of finite index, we can assume $H$ to be Zariski connected.

In what follows, we recall several ingredients in the strategy of Benoist--Kobayashi in \cite{benoist-kobayashi2022tempered2}.
The first ingredient is the existence of nice intermediate subgroups.
\begin{lemma}[{\cite[Lem 4.1]{benoist-kobayashi2022tempered2}}]\label{lem: nice intermediate subgroups}
    There exist two intermediate algebraic subgroups $H\subset F\subset Q\subset G$ with the following properties:
    \begin{enumerate}
        \item $Q$ is a parabolic subgroup of $G$ of minimal dimension containing $H$.
        \item Let $U$ be the unipotent radical of $Q$. There exists a Levi decomposition $Q = LU$ such that $L$ is a maximal reductive subgroup of $Q$ and that $H = (L\cap H)(U\cap H)$.
        \item $S:=L\cap H$ is a maximal reductive subgroup of $H$ and $V:=U\cap H$ is the unipotent radical of $H$.
        \item $F= SU$.
    \end{enumerate}
    Thus, we have a chain of algebraic subgroups with compatible Levi decompositions
    \begin{equation*}
        H=SV \subset F=SU \subset Q= LU \subset G.
        \qedineq
    \end{equation*}
\end{lemma}
The notation of these groups will be standing from \autoref{subsec: ingredients bk} to \autoref{subsec: proof for alg subgp}, and their Lie algebras will be denoted by the corresponding Fraktur letters.
We can suppose that $Q\ne G$, for otherwise the algebraic subgroup $H$ is already reductive and we can conclude by \autoref{mythm: reductive subgp, delta=theta=rho}.

The second ingredient is the domination of group actions. On the homogeneous space $U/V$, the reductive group $S$ acts by conjugation and the unipotent group $U$ acts by left translation.
\begin{lemma}[{\cite[Prop 4.4]{benoist-kobayashi2022tempered2}}]\label{lem: domination of group action}
    Let $U/V$ be equipped with a $U$-invariant Radon measure $\Vol$.
    Then for every compact subset $D\Subset U/V$, there exists a compact subset $D_0\Subset U/V$ such that we have for all $s\in S$ and $u\in U$,
    \begin{equation*}
        \Vol(suD\cap D) \leq \Vol(sD_0\cap D_0).
        \qedineq
    \end{equation*}
\end{lemma}

Since $F/H=U/V$, this lemma states that sufficiently many matrix coefficients of the unitary representation $(\lambda_{F/H}, \Lspace^2(F/H))$ are dominated by those of the unitary representation $(\sigma_0, \Lspace^2(U/V))$ of $F$, whose action is given by $\sigma_0(su) = \lambda_{F/H}(s)$ for all $s\in S$ and $u\in U$, cf.\,\cite[\S 4.1]{benoist-kobayashi2023tempered4}.

Now this majoration carries to induced representations (cf.\,\cite[\S 4.2]{benoist-kobayashi2023tempered4}), which implies that the coefficients of $\lambda_{Q/H}=\Ind_F^Q \lambda_{F/H}$ (\autoref{lem: induction by stage}) are dominated by those of the induced representation $\pi_0:= \Ind_F^Q \sigma_0.$
We have thus obtained the following conclusion, cf.\,also \cite[Prop 4.9]{benoist-kobayashi2022tempered2}.
\begin{lemma}\label{prop: domination of Q/H}
    For any $f_1, f_2\in \Cc(Q/H)$, there exist a nonnegative function $\varphi_0\in \Cc(L)$ and a compact subset $D_0 \Subset U/V$ such that we have uniformly for all $q\in Q$,
    \begin{equation*}
        \Modofinner{\lambda_{Q/H}(q)f_1, f_2} \leq \Innerprod{\pi_0(l_q)\rmI_S^L(\varphi_0, \indicator_{D_0}), \rmI_S^L(\varphi_0,\indicator_{D_0})},
    \end{equation*}
     where $l_q$ is the $L$-component of $q$ in $Q=LU$.
    \qed
\end{lemma}
Note that the right hand side depends only on the restricted representation $\pi_0|_L$.
We claim that $\pi_0|_L \cong \Ind_S^L(\sigma_0|_S)$.
Indeed, since the unitary representation $\sigma_0$ is trivial on the unipotent radical $U$, the induced representation $\pi_0$ is also trivial on $U$.

The last ingredient is more classical.
The reductive group $S$ acts by conjugation on the pair of real unipotent Lie groups $U\supset V$.
By the theory of adapted bases for nilpotent Lie algebras, there exists an $S$-invariant complementary space $\ulie = \wlie \oplus \vlie$ such that the exponential map induces an $S$-equivariant homeomorphism $\exp: \wlie\to U/V$, cf.\,\cite[Lem 4.7]{benoist-kobayashi2023tempered4}.
Hence, the representation $(\sigma_0|_S, \Lspace^2(U/V))$ is equivalent to the unitary representation $(\tau, \Lspace^2(\ulie/\vlie))$ of $S$ derived from the linear representation of $S$ on the quotient space $\ulie/\vlie \cong \wlie$, as in \eqref{eq def: unit rep on L2 of an alg rep}.
Then \autoref{cor: rho fun and volume decay} implies immediately the following estimates.
\begin{corollary}\label{cor: unif decay sigma0}
    Fix a Cartan decomposition $S=K_S A_S K_S$ with the Cartan projection $\kappa_S: S \to \alie_\slie/W_S$.
    For any compact subset $D \Subset U/V$, there exists a constant $C>0$ such that we have uniformly for all $s\in S$,
    \begin{equation*}
        \Modofinner{\sigma_0(s)\indicator_{D}, \indicator_{D}} \leq C \bigexp{-\rho_{\ulie/\vlie}\kappa_S(s)}.
        \qedineq
    \end{equation*}
\end{corollary}

\subsection{Uniform decay of reductive induction}
Next, we inspect the uniform decay of $\pi_0|_L = \Ind_S^L(\sigma_0|_S)$.

First recall some notation from \autoref{subsec: parabolic subgroups}.
The real reductive group $L$ admits the maximal compact subgroup $K_M$ which is contained in $K$ and the Cartan subspace $A$ which is also a Cartan subspace in $G$. The Cartan projection of $L$, denoted by $\kappa_L: L \to \alie$,
is determined by the positive system $\Sigma_M^+$ and $\rho_L = \rho_M$.
The Weyl group $W_L$ acting on $\alie$ is thus identified with a subgroup of $W_G$.

The following decay estimates bring the local volume decay exponent $\betagh$ into play. The proof relies on a refined control of the preliminary decay estimates in \autoref{subsec: transform}.
\begin{lemma}\label{prop: uniform decay of reductive induction}
    Given any function $\varphi\in \Cc(L)$ and compactum $D \Subset U/V$, we form the vector $f = \rmI_S^L(\varphi, \indicator_{D})$ in the representation space of $\pi_0|_L$ identified with $\Ind_S^L(\sigma_0|_S)$. Then there exists a constant $C>0$ such that we have uniformly for all $l\in L$,
    \begin{equation*}
        \Modofinner{\pi_0(l)f, f} \leq C \bigexp{\frac{1}{2}\!\left(-\rho_{\llie} + (2\rhogh-1)\rhog\right)\!(\kappa_L(l))}.
    \end{equation*}
\end{lemma}
\begin{proof}
    By applying \autoref{lem: first majoration of decay of coefficients} to $\pi_0|_L = \Ind_S^L(\sigma_0|_S)$ and \autoref{lem: decompose each CcG into small ones} to $\varphi$, we have, for some compactum $B\Subset L$, uniformly for all $l\in L$ that
    \begin{equation}\label{eq intm: decay of induced S L U/V}
        \Modofinner{\pi_0(l)f, f} \leq C_1 \bigexp{-\rho_\llie\kappa_L(g)} \int_{S\cap BlB} \Modofinner{\sigma_0(s)\indicator_{D}, \indicator_{D}} \dif s.
    \end{equation}
    Next, we inspect the integral term.
    Since the Cartan subspace $\alie_\slie$ of $S$ lies in some Cartan subspace of $L$ which is conjugate by $L$ to the subspace $\alie$, we can identify $\alie_\slie$ with a subspace of $\alie$.
    Write $X=\kappa_L(l)\in\alie$.
    As in \eqref{eq intm: kappa H of BgB} in the proof of \autoref{lem: upper bound on volume growth in H}, we have
    for some uniform constant $r>0$ that
    \begin{equation*}
        S\cap BlB \subset \bigcup_{w\in W_L} K_S \bigexp{\left(wX+\alie(r)\right)\cap \alie_\slie} K_S.
    \end{equation*}
    Combining this with \autoref{cor: unif decay sigma0} yields the following uniform estimates
    \begin{align}
        &\int_{S\cap BlB} \Modofinner{\sigma_0(s)\indicator_{D}, \indicator_{D}} \dif s 
        \leq C_2 \int_{S \cap BlB} \bigexp{-\rho_{\ulie/\vlie}\kappa_S(s)}\dif s\nonumber\\
        &\leq C_2\sum_{w\in W_L} \int_{K_S \bigexp{\left(wX+\alie(r)\right)\cap \alie_\slie} K_S} \bigexp{-\rho_{\ulie/\vlie}\kappa_S(s)} \dif s\nonumber\\
        &\leq C_3 \sum_{w\in W_L}\int_{\left(wX+\alie(r)\right)\cap \alie_\slie}  \bigexp{\rho_\slie(Y)-\rho_{\ulie/\vlie}(Y)} \dif Y,\label{eq intm: sum W int wX+ar}
    \end{align}
    where the last inequality follows from \autoref{prop: KAK Cartan integral}. 
    But since one has the following equalities and inequality of the rho-functions on $\alie_\slie$.
    \begin{align*}
        \rho_\slie - \rho_{\ulie/\vlie} = \rho_\slie +\rho_{\vlie} - \rho_{\ulie} = \rho_{\hlie} - \left(\rho_\glie- \rho_\llie\right)\!/2 \leq \left(\rhogh - 1/2\right)\!\rhog + \rho_\llie/2,
    \end{align*}
    we deduce from the $W_L$-invariance and the uniform Lipschitz property of the rho-functions $\rhog, \rho_\llie$ on $\alie$ that \eqref{eq intm: sum W int wX+ar} is further bounded from above, uniformly for all $l\in L$, by
    $$C_4 \bigexp{\frac{1}{2}\left(\rho_\llie+(2\rhogh-1)\rhog\right)\!(X)}.$$
    We conclude the proof by feeding back to \eqref{eq intm: decay of induced S L U/V}.
\end{proof}

Recall from \autoref{subsec: spherical fn} that the spherical functions of the real reductive group $L=K_M A K_M$ are given for $\chi\in \alie'$ by
\begin{equation*}
    \Xi^L_{\chi}(l) = \int_{K_M} e^{-(\chi+\rho_L)\eta(l^{-1}k_M)} \dif k_M.
\end{equation*}

Let us write $t^+ := \maxde{t,0}$ for $t\in \bbR$.
\begin{corollary}\label{cor: uniform decay of pi0}
    Let $\chi = (2\rhogh-1)^+ \rho\in \alie'$.
    Under the same assumptions as \autoref{prop: uniform decay of reductive induction}, we have uniformly for all $l\in L$,
    \begin{equation*}
        \Modofinner{\pi_0(l)f, f} \leq C\sum_{w\in W_G} \Xi^L_{w\chi}(l).
    \end{equation*}
\end{corollary}
\begin{proof}
    If $\rhogh\leq 1/2$, then \autoref{prop: uniform decay of reductive induction} implies that $\pi_0|_L$ is tempered. We have $\chi=0$ and we conclude by \autoref{thm: temperedness, equivalent formulations}.

    From now on assume $\rhogh > 1/2$.
    By \autoref{eg: rhoh in case of adjoint rep}, we have the equality $\rho_\llie\circ\kappa_L = 2\rho_L\circ\kappa_L$ on $L$ and $\rhog = 2\max_{w\in W_G} w\rho$ on $\alie$. Thus, the uniform estimates of \autoref{prop: uniform decay of reductive induction} yields uniformly for all $l\in L$,
    \begin{align*}
        \Modofinner{\pi_0(l)f, f} \leq C \max_{w\in W_G} \bigexp{\left(-\rho_L + (2\rhogh-1)w\rho\right)\!(\kappa_L(l))}.
    \end{align*}
    But then \autoref{lem: spherical functions} applied to the reductive group $L$ gives the uniform majoration by spherical functions
    \begin{equation*}
        \Modofinner{\pi_0(l)f, f} \leq C \max_{w\in W_G} \Xi^L_{w\chi}(l) \leq C\sum_{w\in W_G}\Xi^L_{w\chi}(l).
        \qedhere
    \end{equation*}
\end{proof}

To summarize, the domination of $\lambda_{Q/H}$ by $\pi_0$ (\autoref{prop: domination of Q/H}) and the decay estimates of $\pi_0$ (\autoref{cor: uniform decay of pi0}) together yield the following uniform decay estimates of the quasi-regular representation $\lambda_{Q/H}$.
\begin{corollary}\label{cor: major of q/h}
    Let $\chi_0 = (2\rhogh-1)^+ \rho$. Then for any $\xi_1,\xi_2\in \Cc(Q/H)$, there exists a constant $C>0$ such that we have uniformly for all $q\in Q$,
    \begin{equation*}
        \Modofinner{\lambda_{Q/H}(q)\xi_1, \xi_2} \leq C \sum_{w\in W_G} \Xi^L_{w\chi_0}(l_q),
    \end{equation*}
    where $l_q$ denotes the $L$-component of $q$ in $Q=LU$. \qed
\end{corollary}

\subsection{Parabolic induction and spherical functions}\label{subsec: proof for alg subgp}
With the premise of \autoref{cor: major of q/h}, the next step of the proof is to establish uniform decay estimates for a unitary representation of $G$ which is induced from the parabolic subgroup $Q$ (\autoref{prop: uniform decay of lambdagh}). This will allow use to control the coefficient decay exponent of algebraic subgroups.

The following lemma is a variant of \cite[Prop 7.14]{knapp1986representation}, highlighting the utility of spherical functions. One should think of both sides of \eqref{eq res: parabolic induction of spherical function} as the matrix coefficients of certain parabolically induced representations.

\begin{lemma}\label{lem: parabolic induction of spherical function}
    Let $B\Subset G$ be a compactum of nonempty interior and let $\chi\in\alie'$ be a real linear form. Then there exists a constant $C=C(B,\chi)>0$, such that uniformly for all $g\in G$, we have
    \begin{equation}\label{eq res: parabolic induction of spherical function}
        e^{-2\rho\kappa(g)}\int_{Q\cap BgB} \Xi^L_{\chi}(l_q^{-1})\, e^{\rhoQ\eta(q)}\dif q \leq C\, \Xi^G_{\chi}(g),
    \end{equation}
    where $\dif q$ is the left Haar measure on $Q$, and $l_q$ denotes the $L$-component of $q\in Q=LU$.
\end{lemma}
\begin{proof}
    Recall the notation from \autoref{subsec: parabolic subgroups}. We have $N_Q=U$ as the unipotent radical of $Q$.
    Since the group $K_M$ normalizes $N_Q$ and the map $\eta(\cdot)$ is $N_Q$-right-invariant, for any $q \in Q$ we have
    \begin{align*}
        \Xi^L_{\chi}(l_q^{-1}) = \int_{K_M} e^{-(\chi+\rho_L)\eta(l_q k_M)} \dif k_M= \int_{K_M} e^{-(\chi+\rho_L)\eta(qk_M)} \dif k_M.
    \end{align*}
    The integral term in the LHS of \eqref{eq res: parabolic induction of spherical function} is thus bounded from above by
    \begin{equation}\label{eq intm: int Q int KM}
        \int_{Q} \int_{K_M} \indicator_{BgB}(q) \,e^{-(\chi+\rho_L)\eta(qk_M)+\rhoQ\eta(q)} \dif k_M \dif q.
    \end{equation}
    Up to enlarging $B$, we can assume the compactum $B$ to be $K$-bi-invariant. Note that the function 
    $$x\mapsto \int_{K_M} \indicator_{BgB}(x) \,e^{-(\chi+\rho_L)\eta(xk_M)+\rhoQ\eta(x)} \dif k_M$$ 
    is $K$-left-invariant and integrable on $G$.
    By the integral formula for the decomposition $G=KQ$ (\autoref{prop: KQ integral, Q psg}), \eqref{eq intm: int Q int KM} equals
    \begin{equation}\label{eq intm: int BgB int KM}
        \int_{BgB} \int_{K_M} e^{-(\chi+\rho_L)\eta(xk_M)-\rhoQ\eta(x)} \dif k_M \dif x.
    \end{equation}
    But since we have $BgB \subset K e^{\kappa(g)+\alie(r)}K$ (\autoref{lem: compact error in cartan projection}) for some constant $r=r(B)$, the integral formula for the Cartan decomposition (\autoref{prop: KAK Cartan integral}) yields that \eqref{eq intm: int BgB int KM} is bounded from above by
    \begin{align}
        &\int_K \int_{\kappa(g)+\alie(r)} \int_K \int_{K_M} e^{-(\chi+\rho_L)\eta(k_1 e^X k_2 k_M)-\rhoQ\eta(k_1 e^X k_2)+2\rho(X)}\dif k_M\dif k_2 \dif X \dif k_1\nonumber \\
        &= \int_{\kappa(g)+\alie(r)} \int_K \int_{K_M} e^{-(\chi+\rho_L)\eta(e^X k k_M)-\rhoQ\eta(e^X k)+2\rho(X)} \dif k_M\dif k\dif X \nonumber \\
        &\leq C_1 e^{2\rho\kappa(g)}\int_{\kappa(g)+\alie(r)} \int_K \int_{K_M} e^{-(\chi+\rho_L)\eta(e^X k k_M)-\rhoQ\eta(e^X k)} \dif k_M\dif k\dif X,\label{eq intm: int kappa(g) int K int KM}
    \end{align}
    uniformly for all $g\in G$. But since $\rhoQ\eta(xk_M) = \rhoQ\eta(x)$ for any $k_M \in K_M$ (\autoref{cor: rhoQeta is KM right invariant}), we deduce from Fubini and $\rho=\rhoQ+\rho_L$ that
    \begin{align*}
        &\int_K \int_{K_M} e^{-(\chi+\rho_L)\eta(e^X k k_M)-\rhoQ\eta(e^X k)} \dif k_M\dif k\\
        &= \int_{K_M} \int_{K} e^{-(\chi+\rho_L)\eta(e^X k)-\rhoQ\eta(e^X k k_M^{-1})} \dif k \dif k_M\\
        &= \int_{K} e^{-(\chi + \rho)\eta(e^X k)} \dif k = \Xi^G_{\chi}(e^X).
    \end{align*}
    Feeding back to \eqref{eq intm: int kappa(g) int K int KM}, we obtain the upper bound on the integral term in the LHS of \eqref{eq res: parabolic induction of spherical function} by
    \begin{equation*}
        C_1 e^{2\rho\kappa(g)} \int_{\kappa(g)+\alie(r)} \Xi^G_{\chi}(e^X) \dif X.
    \end{equation*}
    For any $X\in \kappa(g)+\alie(r)$, there exist some $Y\in \alie(r)$ and $k'\in K$ for which $\eta(e^X k) = \eta(e^Y k') + \eta(e^{\kappa(g)}k)$. The expression and the $K$-bi-invariance of the spherical function yield that uniformly for all $X\in \kappa(g)+\alie(r)$,
    \begin{equation*}
        \Xi^G_{\chi}(e^X) \leq C_2\, \Xi^G_{\chi}(e^{\kappa(g)}) = C_2\,\Xi^G_{\chi}(g),
    \end{equation*}
    whence we obtain the inequality \eqref{eq res: parabolic induction of spherical function}.
\end{proof}

\begin{proposition}\label{prop: uniform decay of lambdagh}
    Let $\chi_0 = (2\rhogh-1)^+\rho$. Then for any functions $f_1, f_2\in \Cc(G/H)$, there exists a constant $C>0$ such that we have uniformly for all $g\in G$,
    \begin{equation*}
        \Modofinner{\lambdagh(g)f_1, f_2} \leq C\,\Xi^G_{\chi_0}(g).
    \end{equation*}
\end{proposition}
\begin{proof}
    Since $\lambdagh = \Ind_Q^G \lambda_{Q/H}$ (\autoref{lem: induction by stage}), we can bound the coefficient on the left hand side from above by passing from $f_i$ to certain positive functions of the form $\rmI_Q^G(\varphi_i, \xi_i)$, with $\varphi_1,\varphi_2\in \Cc(G)$ and $\xi_1,\xi_2\in \Cc(Q/H)$.
    Then \autoref{lem: first majoration of decay of coefficients} and \autoref{lem: decompose each CcG into small ones} yield for some given compactum $B\Subset G$ and uniformly for all $g\in G$ that
    \begin{equation*}\label{eq intm: 1 major of lambdagh}
        \Modofinner{\lambdagh(g)f_1,f_2} \leq C_1e^{-2\rho\kappa(g)} \int_{Q\cap BgB} \Modofinner{\xi_1, \lambda_{Q/H}(q^{-1})\xi_2} \dif\nu_Q(q).
    \end{equation*}
    By \autoref{cor: major of q/h} we have $\Modofinner{\xi_1, \lambda_{Q/H}(q^{-1})\xi_2} \leq C_2 \sum_{w\in W_G} \Xi^L_{w\chi}(l_q^{-1})$, whence uniformly for $q\in Q$,
    \begin{equation*}
        \Modofinner{\lambdagh(g)f_1,f_2} \leq C_3 e^{-2\rho\kappa(g)} \int_{Q\cap BgB} \sum_{w\in W_G} \Xi^L_{w\chi_0}(l_q^{-1})\, e^{\rhoQ\eta(q)}\dif q,
    \end{equation*}
    where we note that the symmetric measure $\dif\nu_Q(q)=e^{\rhoQ\eta(q)}\dif q$.
    Then \autoref{lem: parabolic induction of spherical function} yields uniformly for all $g\in G$,
    \begin{equation*}
        \Modofinner{\lambdagh(g)f_1, f_2} \leq C_4\, \Xi^G_{\chi}(g),
    \end{equation*}
    where we use the fact that $\Xi^G_{w\chi}=\Xi^G_{\chi}$ for any $w\in W_G$ (\autoref{lem: W_G invariance, spherical functions}).
\end{proof}

As a result, we obtain an upper bound of the decay exponent $\thetagh$, and by combining the lower bound obtained in the previous section we have the following result.
\begin{proposition}\label{prop: eq for alg subgps}
    For any algebraic subgroup $H$ of $G$, we have $$ \thetagh = \maxde{\betagh,\, 1/2}.$$
\end{proposition}
\begin{proof}
    By \autoref{prop: uniform decay of lambdagh} and \autoref{lem: spherical functions}, for any $\varepsilon>0$ and any functions $f_1, f_2\in \Cc(G/H)$, we have the uniform decay for $g\in G$,
    \begin{align*}
        \Modofinner{\lambdagh(g)f_1, f_2} \leq C \bigexp{-2\!\left(1-\varepsilon-\maxof{\rhogh, 1/2}\right)\!\rho\kappa(g)}.
    \end{align*}
    By the definition of $\thetagh$, we get $\thetagh\leq \maxof{\rhogh, 1/2}$.
    We conclude thanks to \autoref{prop: beta <= theta} which states that $\betagh\leq \thetagh$.
\end{proof}

\subsection{Uniform and local decay exponents}\label{subsec: extend to connected subgroups}
In this last part, we complete the proof of \autoref{mythm: rho theta delta}.
We start with some definitions and tools from representation theory.
\begin{definition}
    Given $p\in [1,\infty)$, a unitary representation $(\pi,\calV)$ of $G$ is
    \begin{itemize}
        \item \emph{strongly $L^p$} if $\Innerprod{\pi(\cdot)v,v}\in\Lspace^p(G)$ for a dense subspace of $v$;
        \item \emph{almost $L^p$} if it is strongly $\Lspace^{p+\varepsilon}$ for all $\varepsilon>0$;
        \item \emph{totally $L^p$} if all matrix coefficients of $\pi$ are $\Lspace^p$;
        \item \emph{totally $L^{p+}$} if it is totally $\Lspace^{p+\varepsilon}$ for all $\varepsilon>0$.
    \end{itemize}
\end{definition}
\begin{remark*}
    In the literature, the terms \emph{strongly $L^p$} and \emph{almost $L^p$} are widely used (e.g.\,\cite{cowling-haagerup1988almost,shalom2000rigidity}), though their precise meanings vary across authors;
    being \emph{totally $L^p$} often appears in the guise of $A_\pi \subset \Lspace^p(G)$ (e.g.\,\cite{cowling1978the-kunze-stein, samei-wiersma2024exotic}). 
\end{remark*}

\begin{definition}
    Given a unitary representation $\pi$ of $G$, define $p(\pi)$ to be the minimal $p\geq 1$ such that $\pi$ is totally $\Lspace^{p+}$. If such $p$ does not exist (e.g.\,if $\pi=1_G$ with $G$ noncompact), we put $p(\pi)=\infty$.
\end{definition}
\begin{remark*}
    For $G$ noncompact, we always have $p(\pi)\geq 2$ for any unitary representation $\pi$. For semisimple $G$, we have $p(\lambda_{G/H})\geq p_{G/H}$ by definition, but in general it can happen that $p(\lambda_{G/H})=2> p_{G/H}$, e.g.\,if $H$ is compact.
\end{remark*}

Recall from \autoref{subsec: weak containment} that $\sigma\prec \pi$ denotes that $\sigma$ is weakly contained in $\pi$ as unitary representations of $G$.

\begin{lemma}\label{lem: Lp carries to weakly contained}
    If $\sigma\prec \pi$ and $\pi$ is totally $\Lspace^{p}$, then $\sigma$ is also totally $\Lspace^p$. As a result, $p(\sigma)\leq p(\pi)$.
\end{lemma}
\begin{proof}
    By \cite[Lem 27]{kunze-stein1960uniformly}, $\pi$ being totally $\Lspace^p$ is equivalent to the existence of a constant $C_p>0$ such that $\opnorm{\pi(f)}\leq C_p \Norm{f}_{p'}$ for all $f\in\Cc(G)$, where $1/p+1/p'=1$. But $\sigma\prec\pi$ implies $\opnorm{\sigma(f)}\leq \opnorm{\pi(f)}$ (\autoref{fact on weak containment}), whence $\opnorm{\sigma(f)}\leq C_p \Norm{f}_{p'}$ for all $f\in\Cc(G)$, i.e.\,$\sigma$ is totally $\Lspace^p$. 
\end{proof}

\begin{lemma}\label{lem: theta 1 2 herz majoration}
    Let $G$ be a noncompact real semisimple algebraic group. For any closed subgroups $H_2<H_1<G$, we have $\theta_{G/H_2}\leq \theta_{G/H_1}$.
\end{lemma}
\begin{proof}
    Herz's \emph{principe de majoration}, cf.\,\cite[Lem 1.4]{cowling1978the-kunze-stein}, \cite[\S 3.1]{benoist-kobayashi2022tempered2}.
\end{proof}

\begin{lemma}[{\cite[Thm 5.3]{samei-wiersma2024exotic}}]\label{lem: samei-wiersma semisimple}
    For a noncompact real semisimple algebraic group $G$ and $2 \leq p <\infty$, a unitary representation being almost $\Lspace^p$ is equivalent to being totally $\Lspace^{p+}$. Hence, for any closed subgroup $H$,
    \begin{enumerate}
        \item if $p_{G/H}\geq 2$, then $p_{G/H}=p(\lambda_{G/H})$;
        \item if $p(\lambda_{G/H})>2$, then $p_{G/H}=p(\lambda_{G/H})$.\qed
    \end{enumerate}
\end{lemma}

\begin{proposition}\label{lem: coamenable subgps have same exponent}
    Let $G$ be a noncompact real semisimple algebraic group.
    For any closed subgroups $H_2<H_1<G$ with $\Lspace^2(G/H_1)$ non-tempered, if $H_2$ is co-amenable in $H_1$ (meaning $1_{H_1}\prec_{H_1}\lambda_{H_1/H_2}$), then $\theta_{G/H_1}=\theta_{G/H_2}$. In particular, $\delta_{G/H_1}=\delta_{G/H_2}$.
\end{proposition}
\begin{remark*}
    The non-tempered assumption cannot be simply removed.
    For $H_2$ trivial and $H_1$ a minimal parabolic subgroup, the exponents $\theta_{G/H_1}=1/2$ and $\theta_{G/H_2}=0$ are different despite co-amenability.
\end{remark*}
\begin{proof}
    We have proven $\theta_{G/H_1}\geq \theta_{G/H_2}$ in \autoref{lem: theta 1 2 herz majoration}.

    For the other direction, if $H_2$ is co-amenable in $H_1$, i.e.\,$1_{H_1}\prec \lambda_{H_1/H_2}$, then by unitary induction from $H_1$ to $G$ we obtain $\lambda_{G/H_1}\prec \lambda_{G/H_2}$. Thus by \autoref{lem: Lp carries to weakly contained}, we have $p(\lambda_{G/H_1}) \leq p(\lambda_{G/H_2})$. But since $\Lspace^2(G/H_1)$ is non-tempered, \autoref{mythm: theta and delta} implies that $p_{G/H_1} > 2$, but then \autoref{lem: samei-wiersma semisimple} further implies $2 < p_{G/H_1} \leq p(\lambda_{G/H_1}) \leq p(\lambda_{G/H_2}) = p_{G/H_2}$.
    By the equality $\theta_{G/H}=1-1/p_{G/H}$ (\autoref{mythm: theta and delta}), we conclude $\theta_{G/H_1}\leq \theta_{G/H_2}$.
\end{proof}

Finally, we can complete the proof of \autoref{mythm: rho theta delta}.
\begin{proof}[Proof of \autoref{mythm: rho theta delta}]
    Let $H$ be a closed connected subgroup of the real semisimple algebraic group $G$.
    By Chevalley's \emph{théorie des répliques}, there exist two algebraic subgroups $H_1, H_2$ of $G$ which satisfy 
    $H^0_1 \subset H \subset H_2$ and $\hlie_1 = [\hlie, \hlie] = [\hlie_2, \hlie_2]$,
    where $H_1^0$ is the identity component of $H_1$ with $[H_1:H_1^0]<\infty$. Hence we have $\theta_{G/H_1^0}=\theta_{G/H_1}$ and $\beta_{G/H_1^0}=\beta_{G/H_1}$.

    First assume that $\Lspace^2(G/H)$ is tempered, so that $\thetagh\leq 1/2$. As $H$ is co-amenable in $H_2$, we have $\lambda_{G/H_2}\prec \lambda_{G/H}$ by induction, so $\Lspace^2(G/H_2)$ is also tempered and $\beta_{G/H}\leq \beta_{G/H_2} \leq 1/2$ by \autoref{prop: eq for alg subgps}.
    
    Next assume that $\Lspace^2(G/H)$ is non-tempered. 
    Since $H_1$ is co-amenable in $H_2$, we can deduce from \autoref{lem: coamenable subgps have same exponent} and \autoref{prop: eq for alg subgps} that $\beta_{G/H_1}=\theta_{G/H_1}=\theta_{G/H_2}=\beta_{G/H_2}$.
    Now since $H^0_1\subset H \subset H_2$, \autoref{lem: theta 1 2 herz majoration} gives $\theta_{G/H_1}\leq \theta_{G/H}\leq \theta_{G/H_2}$.
    Moreover, since $\hlie_1\subset \hlie\subset \hlie_2$, we have $\beta_{G/H_1}\leq \beta_{G/H} \leq \beta_{G/H_2}$. Hence, $\theta_{G/H}=\beta_{G/H}$ above $1/2$.
\end{proof}

\bibliographystyle{abbrv}
{\footnotesize \bibliography{reflist.bib}}

\vspace*{0.5em}
\noindent{\scshape Y. Benoist}: CNRS, Université Paris-Saclay, Orsay, France\\
Email: \texttt{yves.benoist@cnrs.fr}
\vspace*{0.5em}

\noindent{\scshape S. Liang}: Université Paris-Saclay, Orsay, France\\
Email: \texttt{siwei.liang@universite-paris-saclay.fr}

\end{document}